\documentclass[11pt]{article}
\pdfoutput=1
\usepackage{geometry}   
\usepackage{color}
\usepackage{graphicx}
\usepackage{amssymb}
\usepackage{epstopdf} 
\usepackage{hyperref}
\usepackage{graphicx, amsmath, amsthm, latexsym, amssymb, amsfonts, epsfig, url}
\usepackage{xcolor}

\newtheorem{theorem}{Theorem}[section]
\newtheorem{lem}[theorem]{Lemma}
\newtheorem{prop}[theorem]{Proposition}
\newtheorem{theo}[theorem]{Theorem}

\newtheorem{holes_conj}[theorem]{Spread Out Simplices Conjecture}
\newtheorem{weakholes_conj}[theorem]{Weak Spread Out Simplices Conjecture}
\newtheorem{perms_conj}[theorem]{Acyclic System Conjecture}
\theoremstyle{definition}
\newtheorem{definition}[theorem]{Definition}
\newtheorem{remark}[theorem]{Remark}
\newtheorem{example}[theorem]{Example}

\newcommand \s{\sigma}
\newcommand \D{\Delta}

\title{\textsf{Acyclic systems of permutations and \\ fine mixed subdivisions of simplices}}

\author{\textsf{Federico Ardila\footnote{Department of Mathematics, San Francisco State University, \texttt{federico@sfsu.edu}. \newline  Partially supported by the National Science Foundation CAREER Award DMS-0956178, the National Science Foundation Grant DMS-0801075, and  the SFSU-Colombia Combinatorics Initiative.}
 \qquad Cesar Ceballos\footnote{Institut f\"ur Mathematik, Freie Universit\"at Berlin, 
 \texttt{ceballos@math.fu-berlin.de}. \newline Partially supported by the Proyecto Semilla of the Universidad de Los Andes, the Beca Mazda para el Arte y la Ciencia, and the SFSU-Colombia Combinatorics Initiative. }}}
\date{}

\begin{document}
\maketitle

\begin{abstract}
A fine mixed subdivision of a $(d-1)$-simplex $T$ of size $n$ gives 
rise to a system of~${d \choose 2}$ permutations of $[n]$ on the edges 
of $T$, and to a collection of $n$ unit $(d-1)$-simplices 
inside $T$. Which systems of permutations and which 
collections of simplices arise in this way? The \emph{Spread Out Simplices
Conjecture} of Ardila and Billey proposes an answer to the second question. 
We propose and give evidence for an answer to the first question, the \emph{Acyclic System Conjecture}.

We prove that the system of permutations of $T$ determines the collection of simplices of $T$. This establishes the Acyclic System Conjecture as a first step towards proving the Spread Out Simplices Conjecture. We use this approach to prove both conjectures for $n=3$ in arbitrary dimension.
\end{abstract}

\section{\textsf{Introduction}}

The fine mixed subdivisions of a dilated simplex arise in numerous contexts, and possess a remarkable combinatorial structure, which has been the subject of great attention recently. The goal of this paper is to prove several structural results about these subdivisions.

\begin{figure}[ht]
	\centering
	\includegraphics[width=1\textwidth]{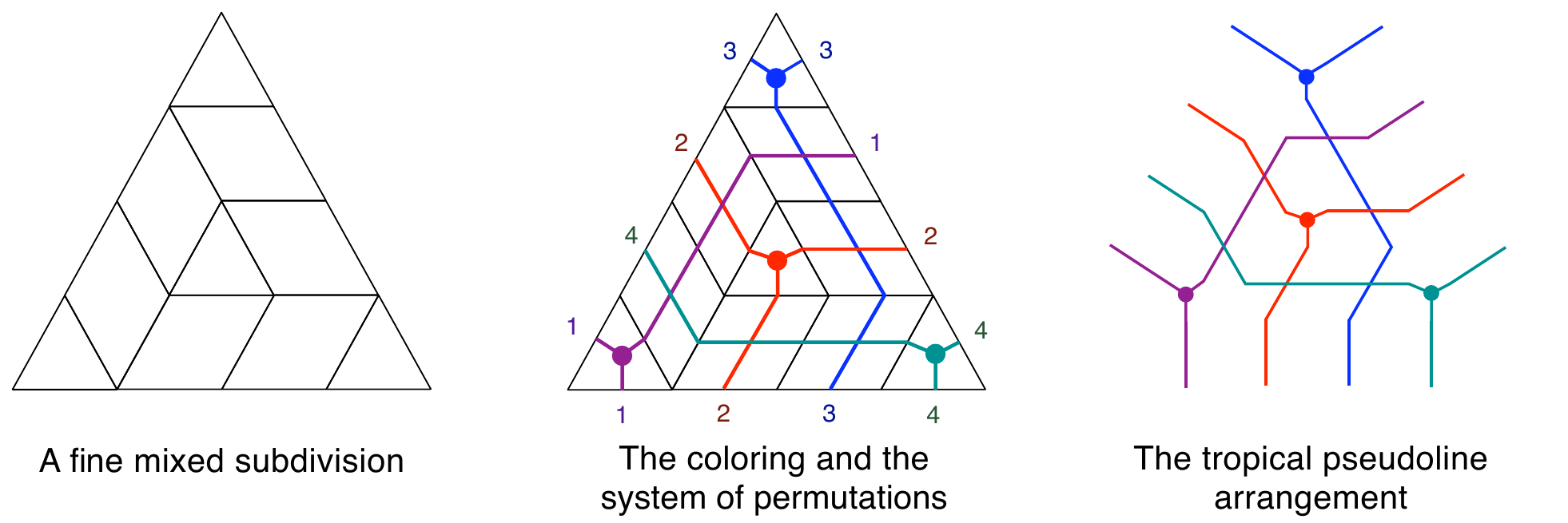}
	\caption{A fine mixed subdivision of $4\Delta_2$. \label{fig_sub_color_hyp}}
\end{figure}

A fine mixed subdivision of a $(d-1)$-simplex $T$ of size $n$ gives 
rise to a system of ${d \choose 2}$ permutations of $[n]$ on the edges 
of $T$, and to a collection of $n$ unit $(d-1)$-simplices 
inside $T$. We address the question: \textbf{Which systems of permutations and which collections of simplices arise from such subdivisions?} We prove several results in this direction. In particular we prove Ardila and Billey's Spread Out Simplices Conjecture \cite[Conjecture 7.1]{[Ardila-Billey]} in the special case $n=3$. 

\smallskip
\noindent
\textbf{1.1. Introduction}. 
We begin by summarizing the different sections of the paper, and stating our main results and conjectures. Figure \ref{fig_sub_color_hyp} illustrates the main concepts with pictures of the case $d=3$. We delay the precise definitions until the later sections. 

\smallskip
\noindent
\textbf{1.2. The fine mixed subdivisions of a simplex} $n\Delta_{d-1}$ are the subdivisions of the dilated simplex $n\Delta_{d-1}$ into fine mixed cells. A \emph{fine mixed cell} is a $(d-1)$-dimensional product of faces of $\Delta_{d-1}$ lying in independent affine subspaces.  For $d=3$, fine mixed subdivisions are the \emph{lozenge} tilings of an equilateral triangle into unit equilateral triangles and rhombi.

The fine mixed subdivisions of $n \Delta_{d-1}$ are in one-to-one
correspondence with triangulations of the polytope
$\Delta_{n-1} \times \Delta_{d-1}$ via the Cayley trick \cite{[Cayley]}. These and other equivalent objects arise very naturally in many contexts
\cite{ABHPS, [Ardila-Develin], Babson, Bayer, [Develin-Sturmfels], Gelfand, Haiman, Orden, Santosflips, Santostoric,  [Cayley], Sturmfels}. 
Fine mixed subdivisions are our main object of study. \textbf{We will often call them simply ``subdivisions".}  

\smallskip
\noindent \textbf{1.3. The coloring of a fine mixed subdivision} 
is a natural coloring of the cells of a fine mixed subdivision.
It gives rise to an arrangement 
of tropical pseudohyperplanes which
plays a key role in the theory of tropical oriented matroids. 
\cite{[Ardila-Develin], HJJS09}.   

\smallskip
\noindent \textbf{1.4. The system of permutations} of a fine mixed subdivision $T$
is the restriction of the coloring to the edges of the simplex $n \Delta_{d-1}$.
It can be seen as a set of permutations of $[n]$, one on each edge. It has the great advantage that it is simpler than the coloring, while maintaining substantial geometric information about the subdivision. Say a system of permutations of $T$ is \emph{acyclic} if no closed walk on the edges of the simplex contains two colors in alternating order: $\ldots i \ldots j \ldots i \ldots j  \ldots i \ldots j \ldots$. In two dimensions, this property characterizes the systems of permutations coming from subdivisions:

\begin{quote}
\textbf{Theorem \ref{theo_characterization_dim2}.} (2-D Acyclic System Theorem)
\emph{A system of permutations on the edges of a triangle
can be achieved by a lozenge tiling if and only if it is acyclic.
}\end{quote}

We also show a result which will be relevant later:

\begin{quote}
\textbf{Theorem \ref{theo_permstoholes}.} (Short version)
\emph{
The positions of the triangles in a lozenge tiling of a triangle 
are completely determined by the system of permutations.
}
\end{quote}

\smallskip
\noindent \textbf{1.5. The Acyclic System Conjecture} seeks to generalize Theorem 
\ref{theo_characterization_dim2} to higher dimensions:

\begin{quote}
\textbf{Theorem \ref{theo_acyclic}.}
\emph{The system of permutations of a fine mixed subdivision of $n\Delta_{d-1}$  is acyclic. 
}\end{quote}

\begin{quote}
\textbf{Acyclic System Conjecture \ref{conj_acyclic}.}
\emph{Any acyclic system of permutations on $n\Delta_{d-1}$ is achievable 
as the system of permutations of a fine mixed subdivision.
}\end{quote}

We remark that this conjecture was recently disproved by Francisco Santos in~\cite{Sa2012}, after this paper was submitted for publication. He constructed a counterexample in the case $n=5$ and $d=4$.

\smallskip
\noindent \textbf{1.6. Duality, deletion, and contraction}
are useful notions, inspired by matroid theory, that were first studied for triangulations of products of simplices by Santos in \cite{[Cayley]}, and for tropical oriented matroids by Ardila and Develin in \cite{[Ardila-Develin]}. We introduce these notions in the context of fine mixed subdivisions and systems of permutations, showing that they are compatible with the earlier ones.

\smallskip
\noindent \textbf{1.7. From systems of permutations to simplex positions}. Ardila and Billey \cite{[Ardila-Billey]} proved that any fine mixed subdivision on $n\Delta_{d-1}$ contains exactly $n$ simplices. We use duality to generalize Theorem \ref{theo_permstoholes} to any dimension:
\begin{quote}
\textbf{Theorem \ref{theorem of positions}.} (Short version)
\emph{
The positions of the $n$ simplices in a fine mixed subdivision of $n\Delta_{d-1}$ are completely determined by its system of permutations.
}
\end{quote}

\smallskip
\noindent \textbf{1.8. The Spread Out Simplices Conjecture} of Ardila and Billey, which is motivated by the Schubert calculus computations of Billey and Vakil \cite{BV}, concerns a surprising relation between fine mixed subdivisions and the matroid of lines in a generic complete flag arrangement. Using the machinery built up in the previous sections, we are able to prove this conjecture for small simplices  in any dimension.

Ardila and Billey \cite{[Ardila-Billey]} showed that the $n$ simplices in any fine mixed subdivision of $n\Delta_{d-1}$ must be \emph{spread out}, meaning that any sub-simplex of size $k$ contains at most $k$ of them. They also conjectured that the converse holds:

\begin{quote}
\textbf{Spread Out Simplices Conjecture \ref{holes_conjecture}.}
\cite{[Ardila-Billey]}
\emph{
A collection of $n$ simplices in $n\Delta_{d-1}$ can be extended to a fine mixed subdivision if and only if it is spread out.}
\end{quote}

\noindent Theorem \ref{theorem of positions} allows us to split the Spread Out Simplices Conjecture \ref{holes_conjecture} into two: the Acyclic System Conjecture \ref{conj_acyclic} and the Weak Spread Out Simplices Conjecture \ref{system_holes_conjecture}:

\begin{quote}
\textbf{Weak Spread Out Simplices Conjecture \ref{system_holes_conjecture}.}
\emph{Every spread out collection of $n$ simplices in $n\Delta_{d-1}$ 
can be achieved as the set of  
simplices of an acyclic system of permutations.
}\end{quote}

\noindent Using this approach, we are able to show:

\begin{quote}
\textbf{Theorem \ref{size 3 federico's conjecture}.}
\emph{
The Spread Out Simplices Conjecture holds for $n=3$.
}
\end{quote}

As we remarked above, Santos \cite{Sa2012} recently disproved the Acyclic System Conjecture. The Spread Out Simplices Conjecture \ref{holes_conjecture} (and the weak version of it) remain open.

\section{\textsf{Fine mixed subdivisions of a simplex}}\label{sec:fine_mixed_subdivisions}

\begin{remark}
\textbf{We will simply refer to fine mixed subdivisions as ``subdivisions"} throughout the paper. The only other subdivisions we will consider are the triangulations of $\Delta_{n-1} \times \Delta_{d-1}$, which we will always refer to as ``triangulations." 
 \end{remark}

\begin{remark} \label{rem:notation}
Throughout the paper, the vertices of $\Delta_{n-1}$ will be denoted $v_1, \ldots, v_n$ and the vertices of $\Delta_{d-1}$ will be denoted $w_1, \ldots, w_d.$ The letters $a$ and $b$ will represent elements of $[d]$ and the letters $i,j,k,$ and $\ell$ will represent elements of $[n]$.
\end{remark}

Fine mixed subdivisions
and several equivalent objects
have been recently studied from many different points of view. 
Aside from their beautiful intrinsic structure  
\cite{Babson, Bayer, Gelfand}, they have been
used as a building block for constructing efficient triangulations of
high dimensional cubes \cite{Haiman, Orden} and disconnected
flip-graphs \cite{Santosflips, Santostoric}. They also arise very
naturally in connection with root lattices \cite{ABHPS}, arrangements of flags \cite{[Ardila-Billey]}, tropical geometry \cite{[Ardila-Develin], [Develin-Sturmfels], SH10}, transportation problems, and Segre embeddings \cite{Sturmfels}. \\

Before defining and studying subdivisions of $n\Delta_{d-1}$ in full generality,
let us start by discussing the easier -- but by no means trivial -- problem of understanding the lozenge tilings of an equilateral triangle. This is the special case $d=3$.

\begin{figure}[h]
	\centering
	\includegraphics[width=0.6\textwidth]{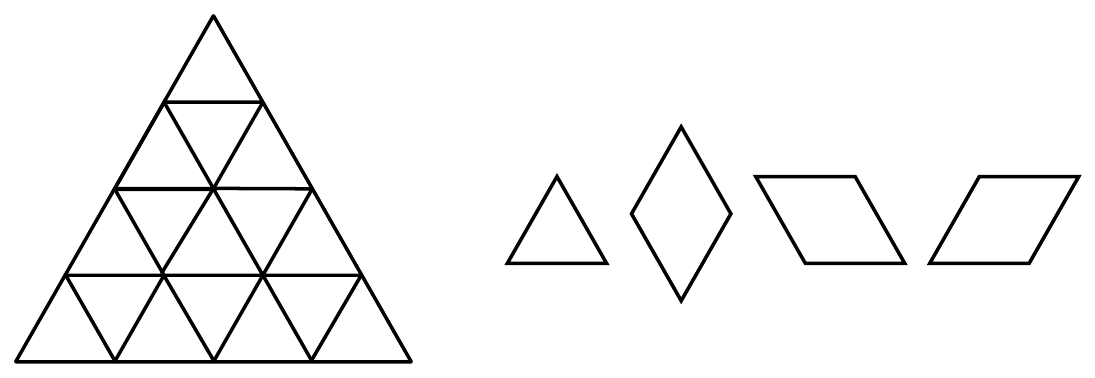}
	\caption{The triangle $4\Delta_2$ and the four different tiles allowed in a lozenge tiling.}
	\label{tiling} 
\end{figure}

Let $n\Delta_2$ be an equilateral triangle with side length 
equal to $n$. A \emph{lozenge tiling} of $n\Delta_2$ is a 
subdivision of $n\Delta_2$ into upward unit triangles and 
unit rhombi, as illustrated on the right hand side of 
Figure \ref{tiling}.  It is not hard to see 
that any lozenge tiling of $n\Delta_2$ consists of $n$ 
triangles and $ n \choose 2 $ rhombi. 
Figure \ref{tiling_T4} shows an example 
of a lozenge tiling of $4\Delta_2$.

\begin{figure}[h]
	\centering
	\includegraphics[width=0.23\textwidth]{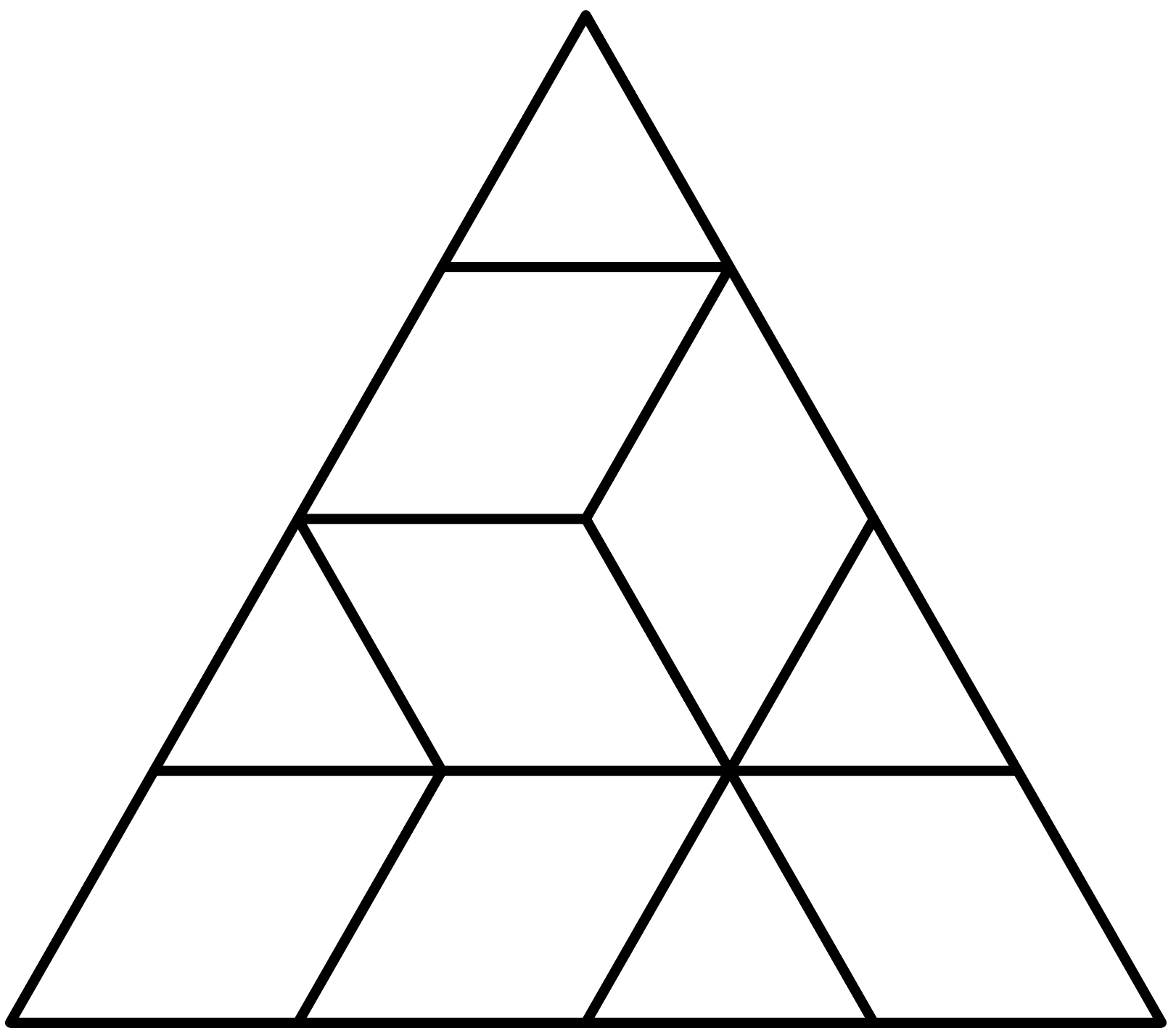}
	\caption{Example of a lozenge tiling of $4\Delta_2$.}
	\label{tiling_T4}
\end{figure}

The most natural high-dimensional analogues of the lozenge tilings of the 
triangle $n\Delta_2$ are the fine mixed subdivisions of 
the simplex $n\Delta_{d-1}$. We briefly recall their definition; for a more thorough treatment, see \cite{[Cayley]}.

The \textit{Minkowski sum} of polytopes $P_1,\dots ,P_k$ 
in $\mathbb R ^m$ is the polytope:
\[
P_1+\dots +P_k:=\{p_1+\dots +p_k\, | \, p_1\in P_1,\dots ,p_k\in P_k\}.
\]
Let 
\[
\Delta_{d-1}=\{ (x_1,\ldots , x_d)\in \mathbb R^d : x_i\geq 0 \text{ and } x_1+\ldots +x_d=1 \}
\]
be the standard unit $(d-1)$-simplex, and $n\Delta _{d-1}=\Delta_{d-1}+\dots +\Delta_{d-1}$ be its scaling by a factor of $n$. 

A \emph{fine mixed cell} is a Minkowski sum $B_1+\dots +B_n$ where the $B_i$s are faces of $\Delta _{d-1}$ which lie in independent affine subspaces, and whose dimensions add up to $d-1$.
A \emph{fine mixed subdivision} $S$ of $n\Delta _{d-1}$ is 
a subdivision
of $n\Delta_{d-1}$ into fine mixed cells, such that the intersection of any two cells is a face of both of them. 
 Figure \ref{subdivision3dim} 
 shows examples of subdivisions of $3\Delta_2$ and $3\Delta_3$.

\begin{figure}[h]
\includegraphics[width=0.3\textwidth]{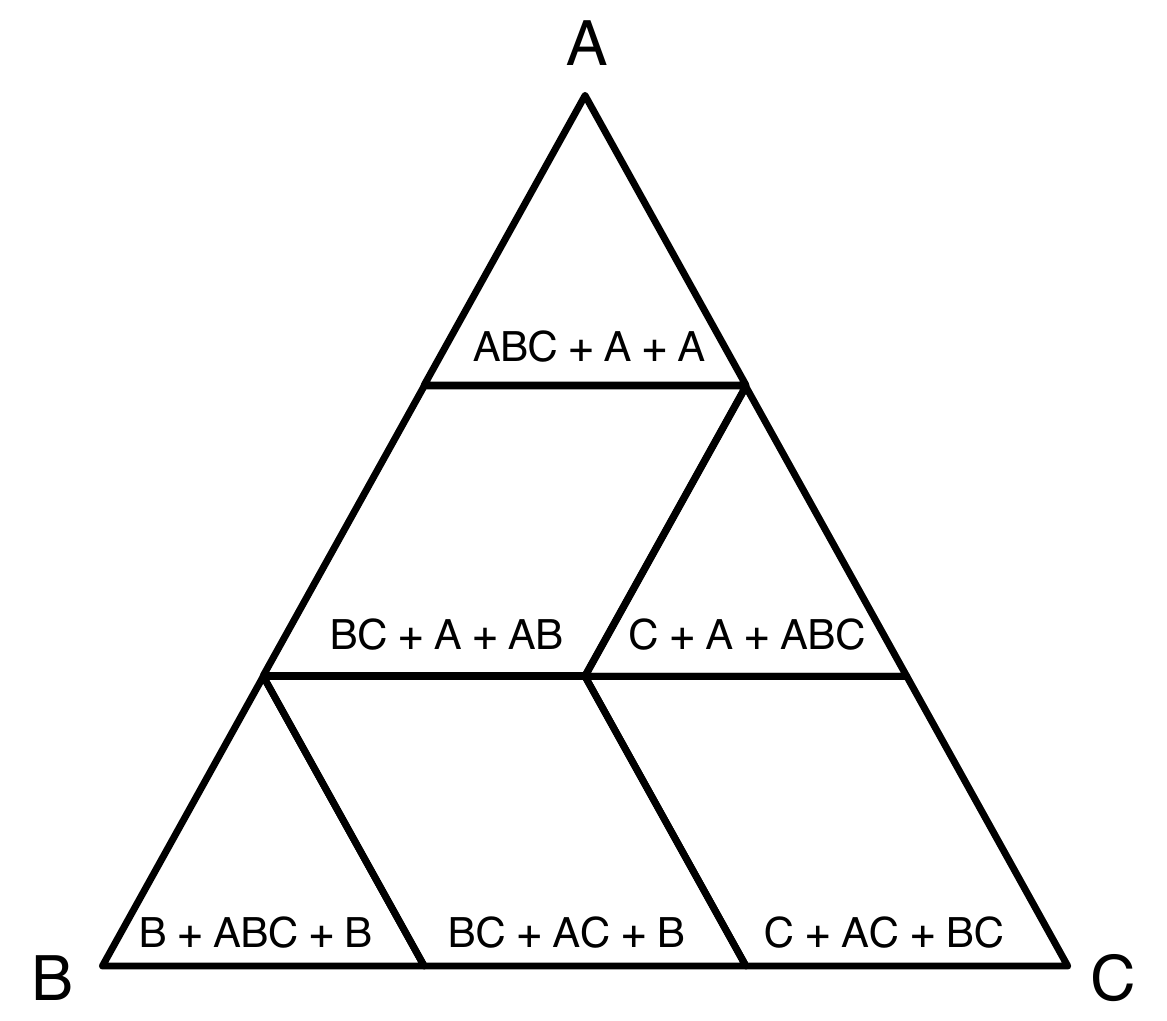} \qquad \qquad  
	\includegraphics[width=0.6\textwidth]{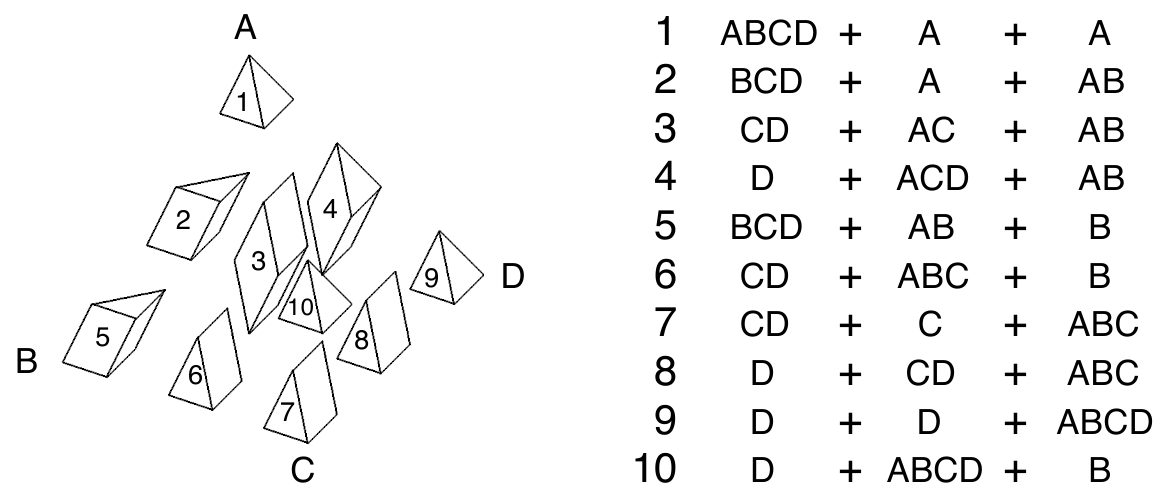}
\caption{Subdivisions of $3\Delta_2$ and $3 \Delta_3$, and the Minkowski sum decompositions of the full dimensional cells.}
	\label{subdivision3dim}
\end{figure}

\begin{remark}
Santos \cite{[Cayley]} showed that the cells in a subdivision of $n\Delta_{d-1}$ can be labeled by ordered Minkowski sums in such a way that, if $B_1 + \cdots + B_n$ is a face of $C_1 + \cdots + C_n$, then $B_i$ is a face of $C_i$ for each $i$. This property is normally required in the definition of a mixed subdivision of $P_1 + \cdots + P_n$, but it holds automatically in this special case.
\end{remark}

\begin{remark}
Ardila and Billey showed that any subdivision of $n\Delta_{d-1}$ 
contains exactly $n$ tiles that are simplices \cite[Proposition 8.2]{[Ardila-Billey]}. 
\end{remark}

\section{\textsf{The coloring of a fine mixed subdivision}}\label{section_The coloring of a fine mixed subdivision}
Given a fine mixed subdivision of a simplex 
one can construct its \emph{colored dual polyhedral complex},
which we simply call its \emph{colored dual}.
This complex can be regarded as a \emph{tropical 
pseudo-hyperplane arrangement}; the interested reader is referred to \cite{[Ardila-Develin], SH10}. 

\begin{figure}[h]
	\centering
	\includegraphics[width=0.3\textwidth]{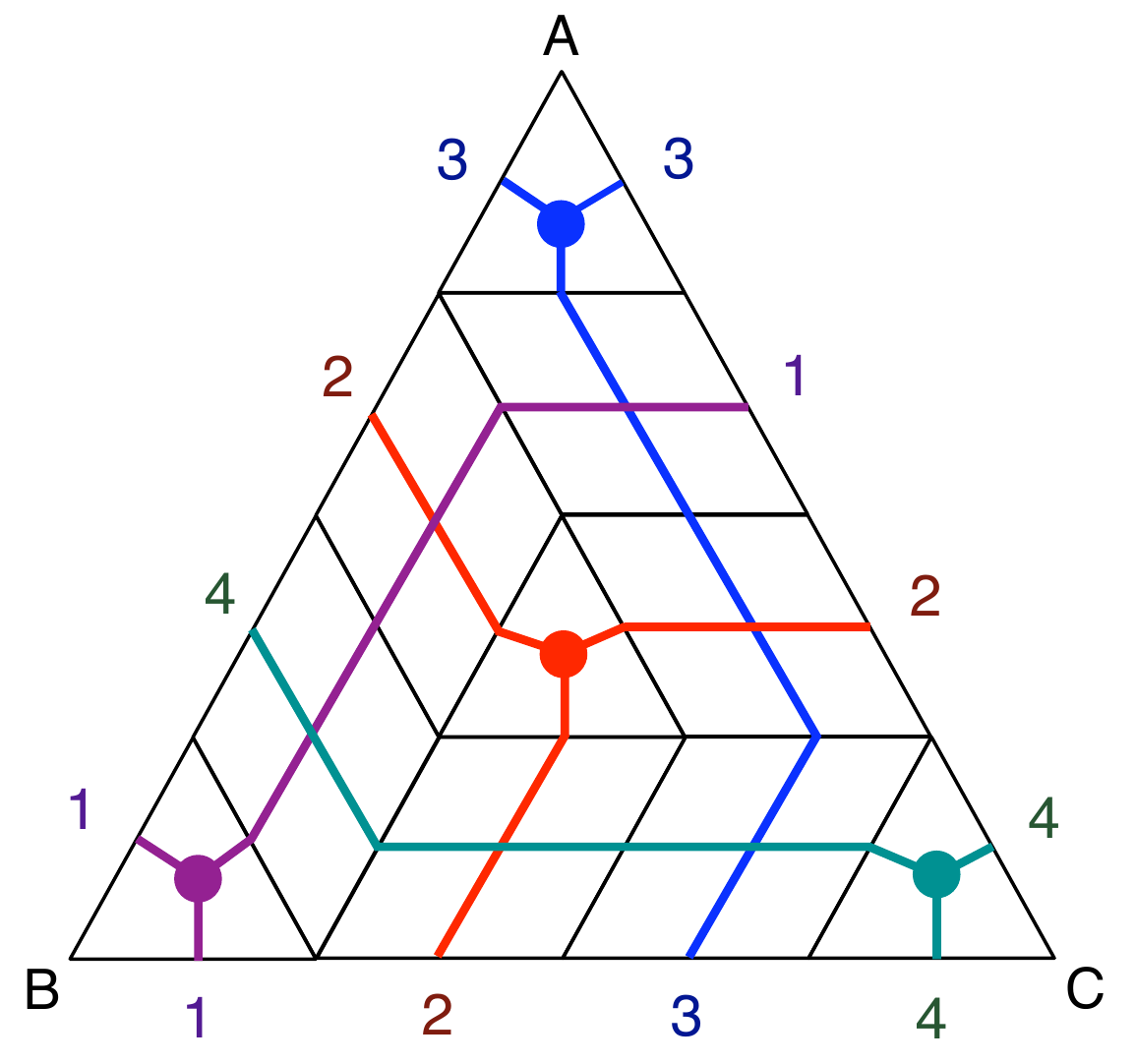}
	\caption{The colored dual of a subdivision of $4\Delta_2$ and 
	the corresponding system of permutations on the edges of the triangle (which we will introduce later).}
\label{Coloration of a triangle}
\end{figure}

Very loosely speaking, the colored dual assigns a different color to each of the $n$ unit simplices, and lets each color spread from the center of the simplex through the cells of the tiling. Figures \ref{Coloration of a triangle} and \ref{coloration3dim} illustrate this process in dimensions 2 and 3.
Formally, we define the coloring using the \emph{mixed Voronoi subdivision}.

\begin{figure}[h]
	\centering
	\includegraphics[width=0.8\textwidth]{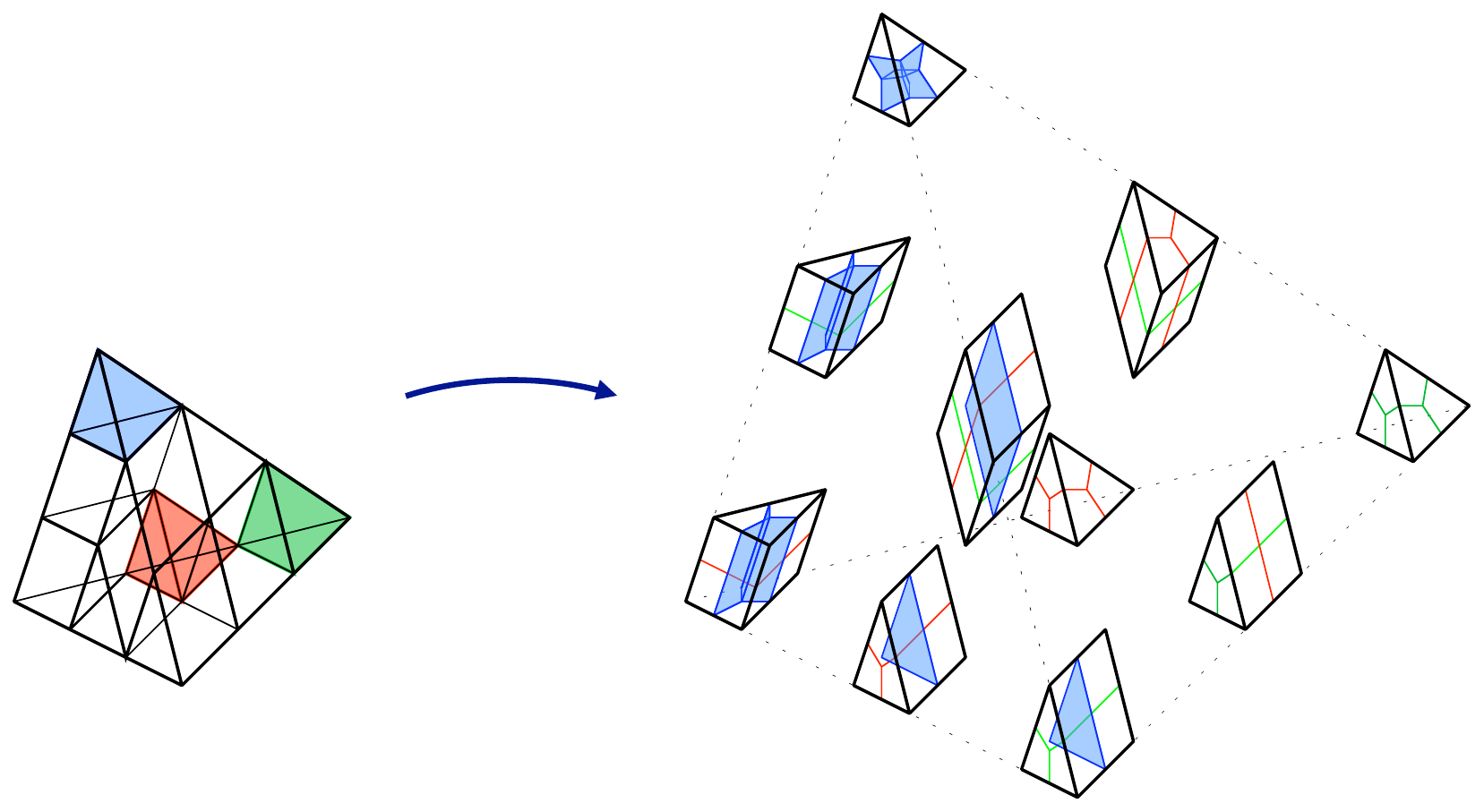}
	\caption{The colored dual of a subdivision of $3\Delta_3$. Only one color is shown in its entirety.}
\label{coloration3dim}
\end{figure}

\begin{definition} The \emph{Voronoi subdivision} of a $k$-simplex divides it 
into $k+1$ regions, where region $i$ consists of the 
points in the simplex for which $i$ is the closest vertex.
Given a subdivision $S$ of $n\Delta_{d-1}$, 
we subdivide each fine mixed cell $S_1+\dots  + S_n$ 
into regions $R_1+\dots +R_n$, where $R_i$ is a 
region in the Voronoi subdivision of $S_i$.
The resulting subdivision of $n\Delta_{d-1}$ is called
the \emph{mixed Voronoi subdivision} of $S$.  
\end{definition}

\begin{figure}[h]
	\centering
	\includegraphics[width=0.9\textwidth]{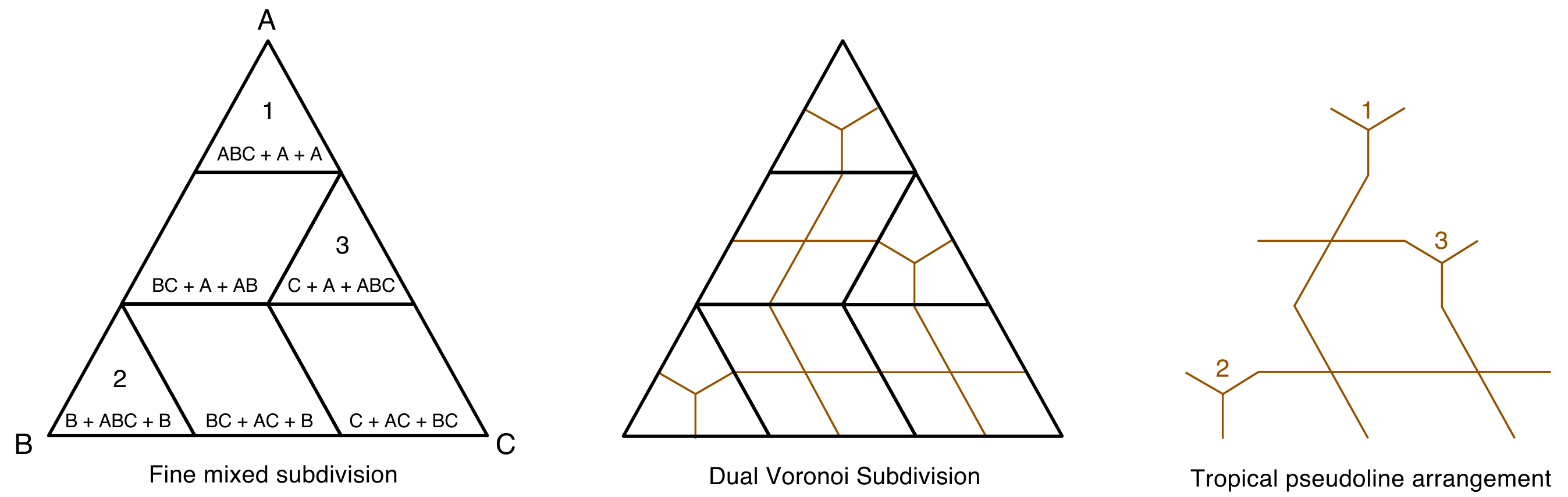}
	\caption{A fine mixed subdivision, its mixed Voronoi subdivision, and
	its colored dual.}
	\label{fine_mixed_voronoi}
\end{figure}

\begin{definition}
The \emph{colored dual} of a subdivision $S$ of $n\Delta_{d-1}$
is the colored polyhedral complex formed by the lower-dimensional faces  
in the mixed Voronoi subdivision of $S$, excluding those on the boundary of $n \Delta_{d-1}$. In the mixed cell $S_1+\dots +S_n$,
\emph{color} $i$ is given to  
$S_1+\dots  \overline {S_i}+\dots  +S_n$,
where $\overline{S_i}$ is the complex of lower dimensional faces  
in the Voronoi subdivision of $S_i$, 
excluding those on the boundary of $S_i$.  
\end{definition}

We will not use the metric properties of this subdivision. In fact,  the dual complex may be defined as a purely combinatorial object. The distinction is not important for us, but this choice of geometric realization will simplify some of our definitions.

\begin{remark}
The \emph{color} $i$ is the subcomplex (called a \emph{tropical pseudohyperplane} in \cite{[Ardila-Develin]}) consisting of the cells of the colored dual having color $i$. It  
subdivides the simplex $n\D_{d-1}$
into $d$ regions which are naturally labeled by the
vertices $w_1,\dots , w_d$ of the simplex.

A cell $S_1 + \cdots + S_n$ is intersected precisely by the colors $i$ such that $\dim (S_i)\neq 0$, or equivalently, $S_i$ has at least two letters. The summand $S_i$ is given by the set of regions (letters) of color $i$  that this cell intersects. 
(See Figure \ref{fine_mixed_voronoi} for an example). It is useful to assign to a face of the triangulation  the same color(s) as its dual cell in the colored dual. 
\end{remark}

\begin{remark}\label{rem:crossing}
In the colored dual of a lozenge tiling $S$, every pair of colors intersects exactly once. One way to see this is to consider, for positive variables $\lambda_1, \ldots, \lambda_n$, the subdivision $S_{\bf \lambda}$ of the triangle $\lambda_1 \Delta_2 + \ldots + \lambda_n \Delta_2$ which is combinatorially isomorphic to $S$. The area of this triangle is $\sum_i \frac12 \lambda_i^2 + \sum_{i<j}  \lambda_i \lambda_j$. The $i$th triangle of $S_\lambda$ contributes an area of $\frac12 \lambda_i^2$ to the subdivision, while a rhombus where colors $i$ and $j$ intersect contributes an area of $\lambda_i \lambda_j$. Therefore there is exactly one such rhombus for each $i$ and $j$.
A similar statement (and proof) holds in any dimension. 
\end{remark}

\section{\textsf{The system of permutations: the two-dimensional case}} \label{section system 2D}

In this section, which focuses only on the two-dimensional case, we define our main object of study: the system of permutations of a subdivision. We then prove several structural results about these systems of permutations. In Sections \ref{sec:acyclic_conjecture}, \ref{section dual}, and \ref{sec:simplex_positions} we will do this (somewhat less successfully) in higher dimensions. 

Given a lozenge tiling of the triangle $n\Delta_2$, and a numbering of its $n$ triangles,
restrict the colored dual to the edges. This determines  three permutations of $[n]$, which we read in clockwise 
direction, starting from the lower left vertex. In Figure 
\ref{Coloration of a triangle}, the system of permutations is 
$(1423,  3124, 4321)$.
We address three questions: 

1. Is a lozenge tiling completely determined by its system of permutations?

2. Which triples of permutations of $[n]$ can arise from a lozenge tiling in this way?

3. How are the system of permutations and the positions of the unit triangles related?

%

The answer to Question 1 is negative, as Figure \ref{fig:dif_tilings_same_system}
shows. Questions 2 and 3 are more interesting, and they are addressed in the following three subsections. 
We answer Question 2 positively: In Theorem 
\ref{theo_characterization_dim2} we give a simple
characterization of the systems of permutations that can be obtained 
from a lozenge tiling. We also answer Question 3 by showing, in Theorem \ref{theo_permstoholes}, that the system of permutations determines uniquely the numbered positions of the triangles. 
 
\begin{figure}[ht]
	\centering
	\includegraphics[width=0.53\textwidth]{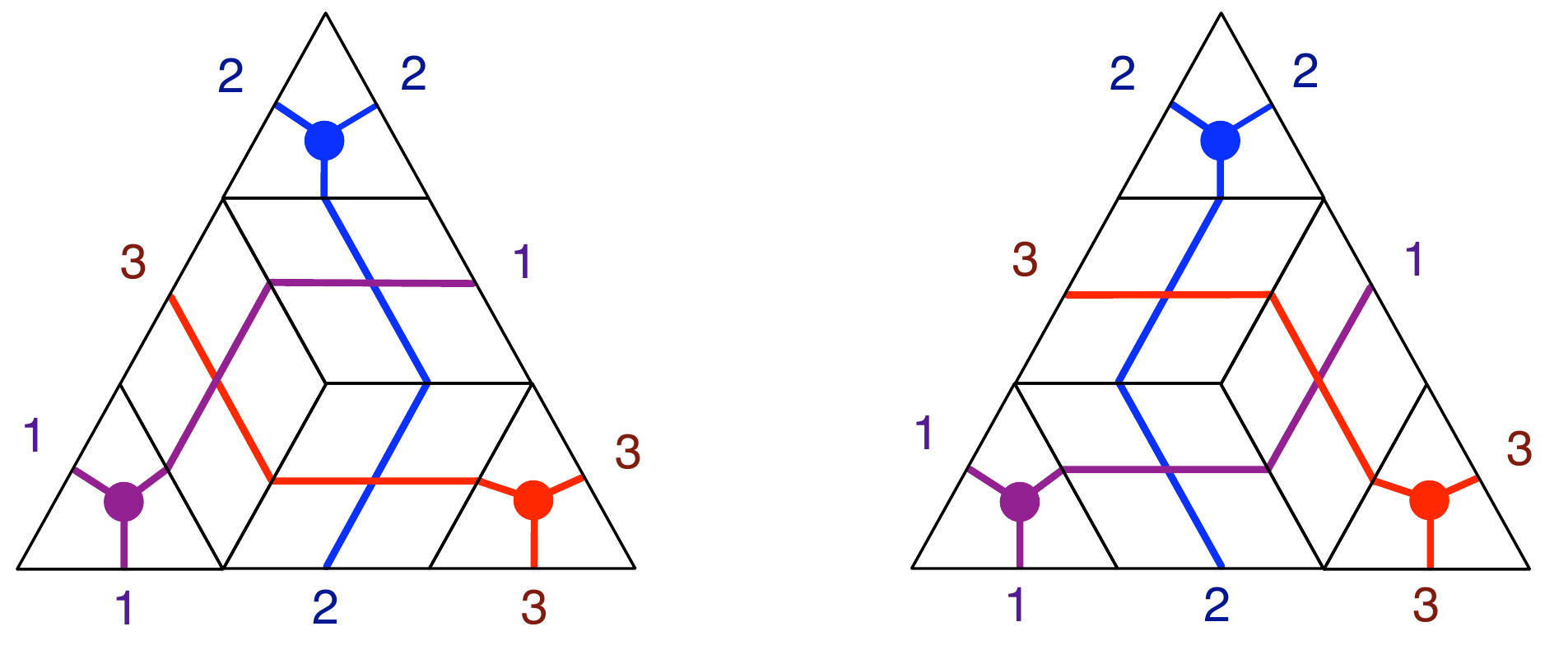}
	\caption{Two different tilings with the same system of permutations.}
	\label{fig:dif_tilings_same_system}
\end{figure}

\subsection{\textsf{The two-dimensional acyclic system theorem.}}

\begin{definition}
A \emph{system of permutations} on the edges of the triangle $n\Delta_2$
is a set of three permutations of $[n]$ on the edges of the triangle.
We say that a system of permutations is \emph{acyclic} if, 
when we read the three permutations in clockwise 
direction, starting from a vertex of the triangle, we never see 
a ``cycle" of the form
$\dots  i \dots  j \dots , \dots  i \dots  j \dots , \dots  i \dots  j \dots $.
\end{definition}

\noindent The system of permutations $(12,12,12)$ on the
edges of $2\D_2$ is the smallest system that is not acyclic. It clearly cannot be realized as the system of permutations 
of a lozenge tiling.

\begin{theo}[$2$-D Acyclic System Theorem]\label{theo_characterization_dim2}
Let $\s$ be a system of permutations on the edges of the triangle $n\Delta_2$. Then $\s$ is achievable as the system of permutations of
a lozenge tiling if and only if $\s$ is acyclic.
\end{theo}

\begin{proof}[Proof of Theorem \ref{theo_characterization_dim2}]

Let $\s=(u,v,w)$ be a system of permutations of a lozenge tiling of the triangle $n\Delta_2$.
If the three permutations $u, v$, and $w$ contained the elements $i$ and $j$ in the same order, then in the dual complex, colors $i$ and $j$ would need to intersect at least twice, contradicting Remark \ref{rem:crossing}. This proves the forward direction. 

For the converse we proceed by induction. The case $n=1$ is trivial. Now assume that the result is true for $n-1$, and consider an acyclic system $\sigma = (u,v,w)$ of permutations of $[n]$. Let $\sigma' = (u',v',w')$ be the acyclic system of permutation of $[n-1]$ obtained by removing the number $n$ from $u,v,$ and $w$, and let $T'$ be a tiling of the triangle $ABC$ (of side length $n-1$) realizing $\sigma'$. Let $D, E,$ and $F$ be the points on the segments $BC, CA,$ and $AB$ where the number $n$ needs to be inserted in the permutations $w', v',$ and $u'$. 

\begin{figure}[h]
	\centering
	\includegraphics[width=0.9\textwidth]{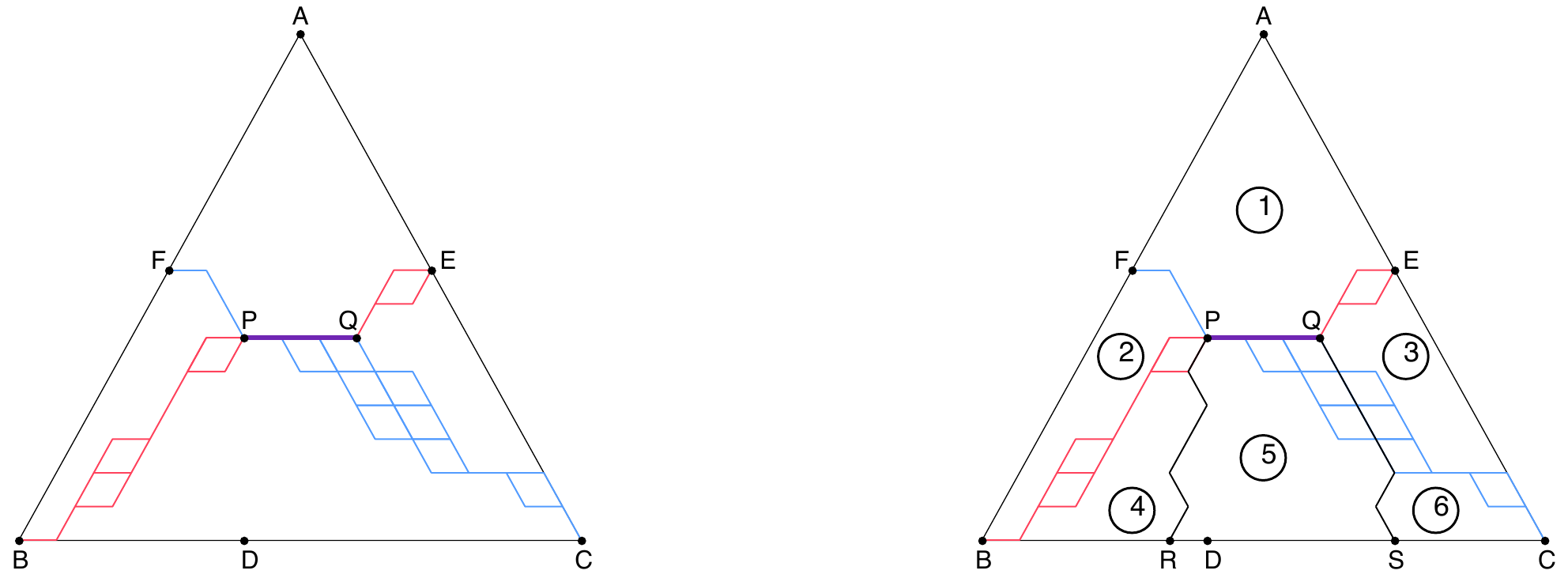}
	\caption{Left: The southwest paths from $E$ to $B$ and the southeast paths from $F$ to $C$. Right: The paths $PR, QS, G_B,$ and $G_C$ divide the triangle into six regions.}
	\label{fig:twopathsacross}
\end{figure}

In the tiling $T'$, let $G_B$ be the union of all ``southwest" paths from $E$ to $B$, consisting of southwest and west edges. Let $G_C$ be the union of all ``southeast" paths from $F$ to $C$, consisting of southeast and east edges. The (non-empty) intersection of $G_B$ and $G_C$ is  a horizontal segment (or possibly a single point); label its left and right endpoints $P$ and $Q$ respectively. Assume that $T'$ was chosen so that the length of $PQ$ is maximum.

Now consider the leftmost ``south" path, using southwest and southeast edges,  from $P$ to edge $BC$. Let its other endpoint be $R$, and call this path $PR$. Similarly, let $QS$ be the rightmost ``south" path from $Q$ to edge $BC$. The previous paths split the triangle into six regions, which we number $1, \ldots, 6$ as shown in the right panel of Figure \ref{fig:twopathsacross}. The cells in $G_C$ and $G_B$ (which are necessarily rhombi) are considered to be in none of the six regions.

If $D$ is between $R$ and $S$, then there is a ``north" path, using northeast and northwest edges, from $D$ to a point $M$ on $PQ$. Consider any northwest path $MF$ along $G_B$ and any northeast path $ME$ along $G_C$. Now cut the tiling along the paths $MD, ME,$ and $MF$, and glue it back together using an equilateral triangle at $M$ and three paths of rhombi of the shape of $MD, ME,$ and $MF$, as shown in Figure \ref{fig:faultlines}. The result will be a tiling $T$ which realizes the system of permutations $\sigma$.

\begin{figure}[h]
	\centering
	\includegraphics[width=0.7\textwidth]{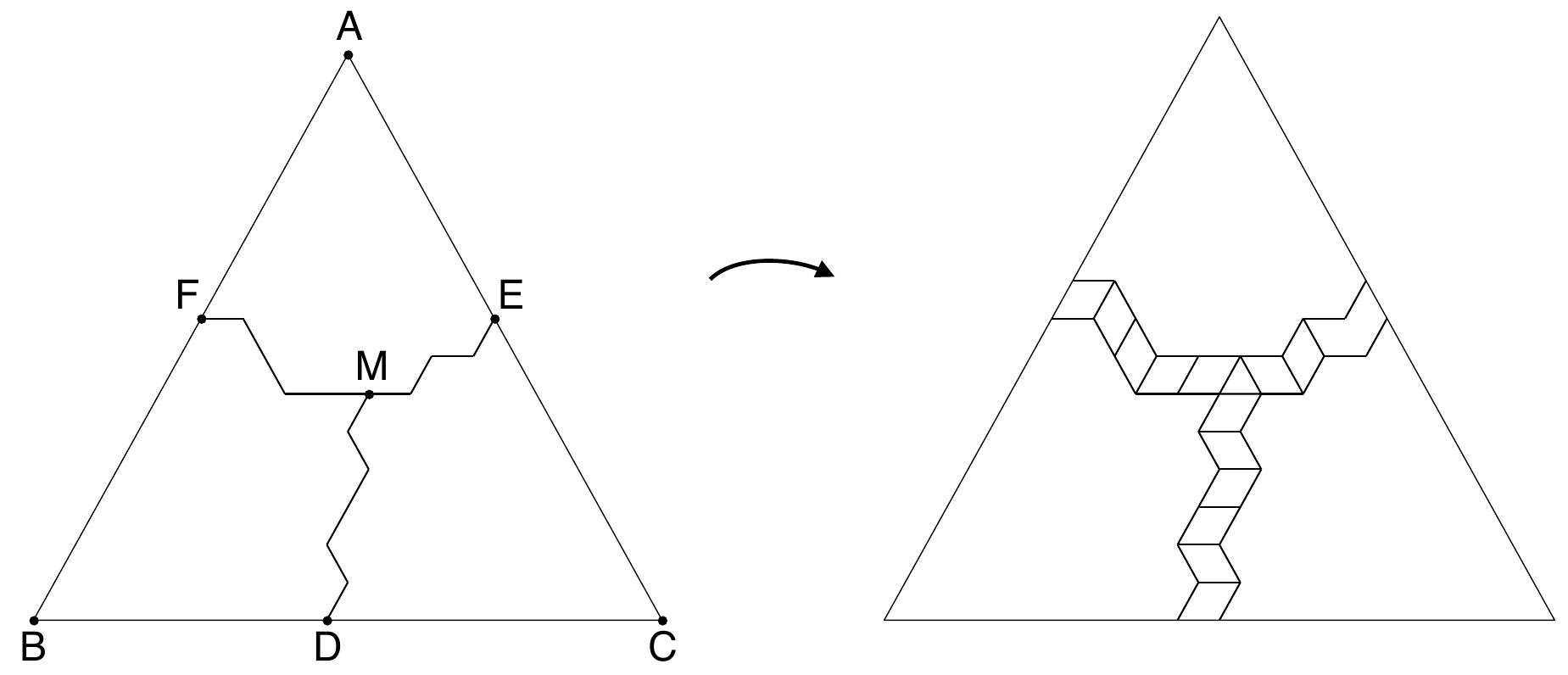}
	\caption{From a tiling of $(n-1)\Delta_2$ to a tiling of $n\Delta_2$.}
	\label{fig:faultlines}
\end{figure}

If $D$ is not between $R$ and $S$, then we claim that $\sigma$ is not acyclic. To prove it, assume without loss of generality that $D$ is to the left of $R$. Consider the edge of $T'$ directly to the left of $R$; say it has color $i$. It is clear that triangle $i$ must be in region $1, 2,$ or $4$. We will show that in fact triangle $i$ is in region $2$. This will imply that $\sigma$ contains the cycle $\dots i \dots n \dots, \dots i \dots n \dots, \dots i \dots n \dots$.

\begin{figure}[h]
	\centering
	\includegraphics[width=0.25\textwidth]{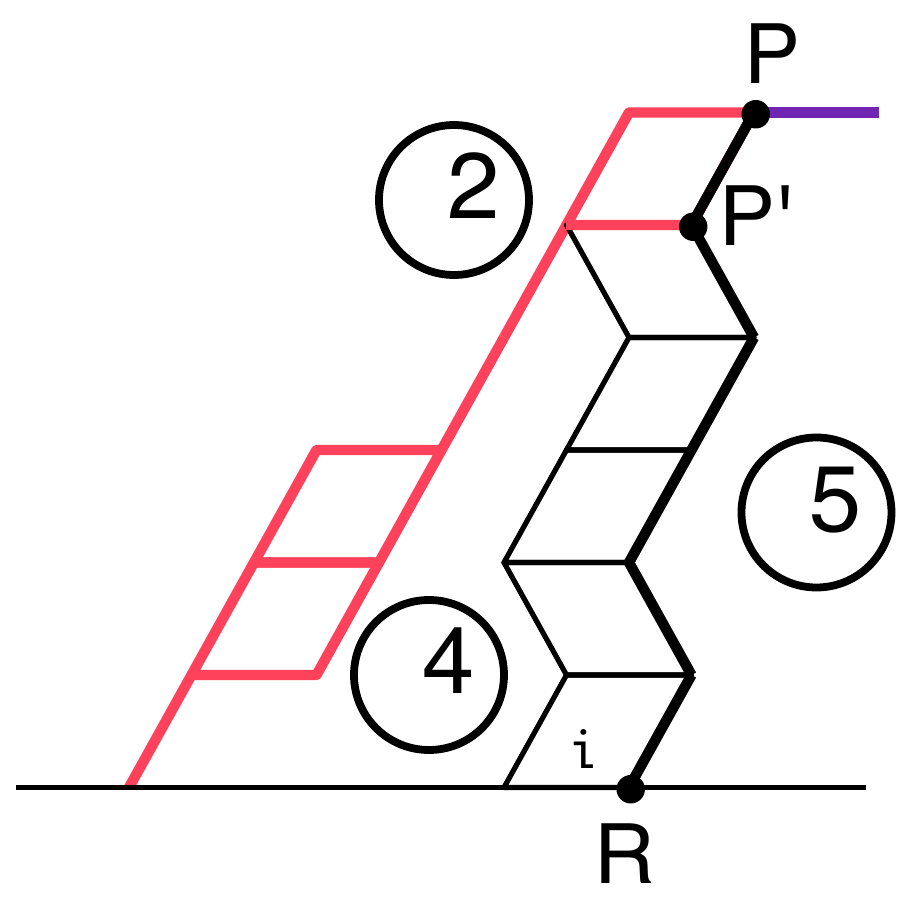}
	\caption{Triangle $i$ cannot be in region 4.}
	\label{fig:notin4}
\end{figure}

First we show that triangle $i$ is not in region $4$. Let $PB$ be the southernmost path from $P$ to $B$ in $G_B$, and let $P'$ be the first place where the paths $PB$ and $PR$ diverge. Note that all edges from $P$ to $P'$ must be southwest edges. By the definitions of $PB$ and $PR$, there are no southwest edges hanging from $P'$. In particular, the next edge in $PB$ after $P'$ is a west edge. This forces all the tiles directly to the left of $PR$ to be horizontal. Therefore triangle $i$ is above $P'$, and hence not in region $4$.

\begin{figure}[h]
	\centering
	\includegraphics[width=0.6\textwidth]{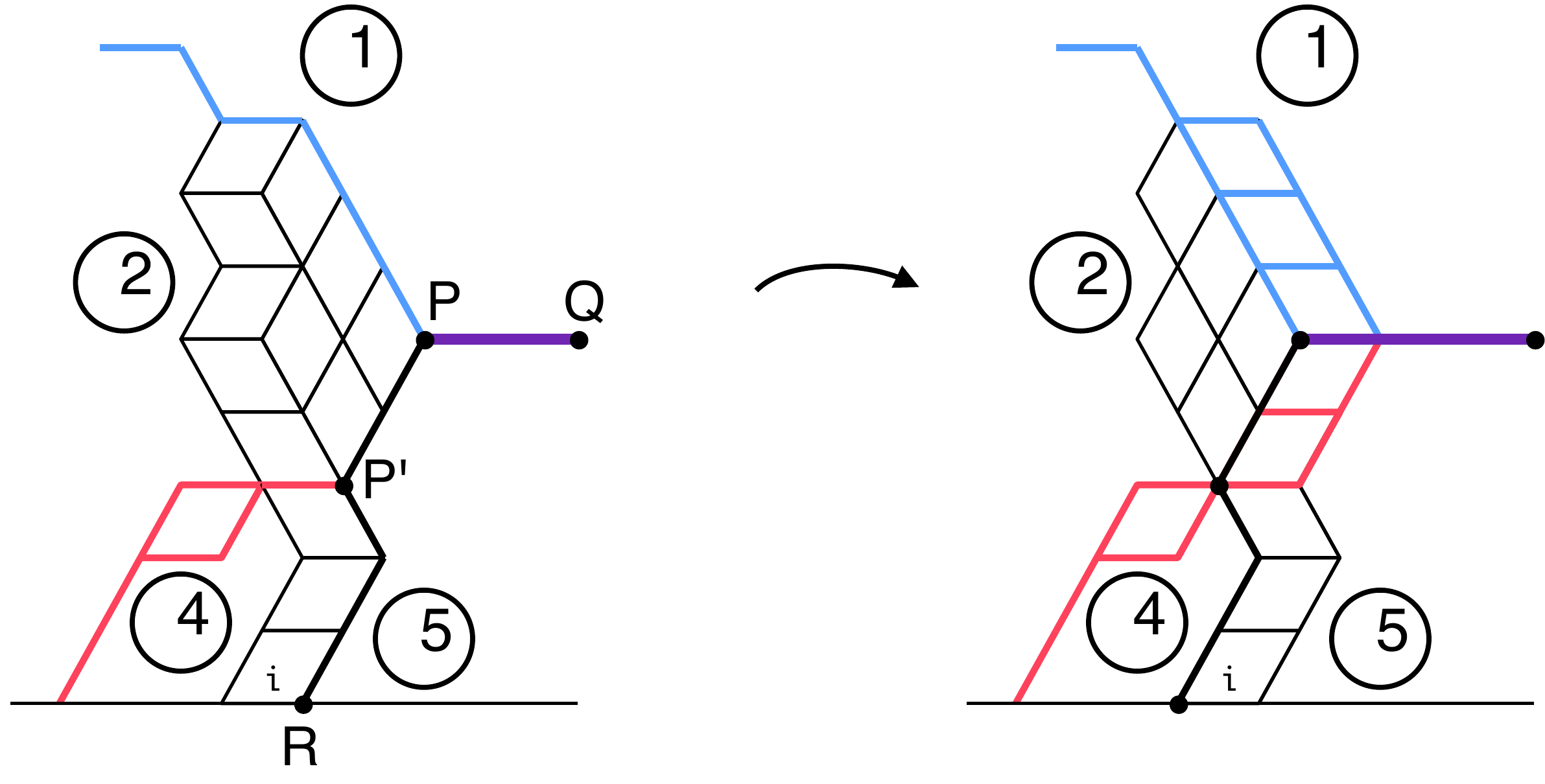}
	\caption{Triangle $i$ cannot be in region 1.}
	\label{fig:notin1}
\end{figure}

Now assume that triangle $i$ is in region $1$. Then color $i$ must enter and exit region $2$ by crossing horizontal edges. This forces all tiles in region $2$ and to the right of color $i$ to be vertical rhombi. But then we can retile this subregion by moving all horizontal rhombi to the right end of region $2$ and shifting all vertical rhombi one unit to the left (as illustrated in Figure~\ref{fig:notin1}). This results in a new tiling of the triangle of side length $n-1$ which has the same system of permutations, but  where the length of $PQ$ is larger, a contradiction. 

It follows that triangle $i$ is in region 2 as desired. This concludes the proof.
\end{proof}

\subsection{\textsf{Acyclic systems of permutations and triangle positions}}

The following is the main result of this section.

\begin{theo}[Acyclic systems of permutations and triangle positions]\label{theo_permstoholes}
In a lozenge tiling of a triangle, the acyclic system of permutations determines uniquely the numbered positions of the unit triangles. Conversely, the numbered positions of the triangles and one permutation of the system determine uniquely the other two permutations.
\end{theo}

\subsubsection{\textsf{A permutation factorization}} \label{sec:perm.fact}
We begin by introducing a way of factoring a permutation uniquely into a particular standard form. This factorization will play an important role in our analysis of lozenge tilings. We use cycle notation for permutations throughout this section. 

\begin{lem}\label{perm.fact}
Every permutation $u$ of $[n]$ can be written uniquely in the form
\[
	u	= (n,\dots ,p_n)\circ\dots \circ(2,3, \dots ,p_2)\circ(1,2, \dots ,p_1)
\]
for integers $p_1, \ldots, p_n$ such that $i \leq p_i \leq n$ for all $i$. 
\end{lem}

\begin{proof}
We proceed by induction on $n$. The case $n=1$ is trivial. Consider a permutation $\pi$ of $[n]$. For the equation to be true, we must have $u^{-1}(1)=p_1$, so $p_1$ is determined by $u$. Then we see that $u \circ (p_1,\ldots,2,1) $ leaves $1$ fixed, and can be regarded as a permutation of $[2, \ldots, n]$. By the induction hypothesis, it can be written uniquely as
\[
u \circ (p_1,\ldots,2,1) = (n,\ldots,p_n) \circ (n-1, \ldots, p_{n-1}) \circ \cdots \circ (2, \ldots, p_2),
\]
for $i \leq p_i \leq n$. This gives the unique such expression for $u$. 
\end{proof}

\begin{lem}\label{perm.fact2}
Similarly, every permutation $v$ of $[n]$ can be written uniquely in the form
\[
v = (1,\dots ,q_1)\circ\dots \circ(n-1, n-2, \dots ,q_{n-1})\circ(n, n-1, \dots ,q_n).\\
\]
for integers $q_1, \ldots, q_n$ such that $i \geq q_i \geq 1$ for all $i$.
\end{lem}

The following lemma tells us how to compute the values of $p_1,\dots , p_n$ and $q_1,\dots , q_n$ in terms of the permutations $u$ and $v$.

\begin{lem}\label{lem:positions_from_u_v}
In the two lemmas above we have
\begin{eqnarray*}
p_k= k + | \{ \ell>k : u^{-1}(\ell) < u^{-1}(k) \}|  \\
q_k = k - | \{ \ell<k : v^{-1}(\ell) > v^{-1}(k) \}| 
\end{eqnarray*}
\end{lem}

\begin{proof}
We prove the result for $p_k$; the proof for $q_k$ is analogous. 
The inverse of $u$ is
\[
	u^{-1}	= (p_1,\dots ,2,1)\circ (p_2,\dots , 3,2) \circ \cdots \circ (p_n,\dots ,n).
\]
Let 
$
\pi_k := (p_k, \ldots, k)  \circ (p_{k+1}, \ldots, k+1) \circ  \dots \circ (p_n, \ldots, n),
$
 a permutation of $\{k,\dots , n\}$. Note that, to obtain $\pi_k$ from $\pi_{k+1}$ (which is a permutation of $\{k+1, \ldots, n\}$), we simply insert $p_k$ at the beginning of the permutation, and substract $1$ from all entries less than or equal to it. For instance, if $(p_1, \ldots, p_6) = (3,2,4,5,6,6)$, then the permutations $\pi_6, \ldots, \pi_1$ are $6, 65, 564, 4563, 24563, 314562$, respectively. In particular, the relative order of $\pi_{k+1}(a)$ and $\pi_{k+1}(b)$ (where $a, b \geq k+1$) is preserved in $\pi_k$.
 
It follows that, for $\ell>k$, we have $ u^{-1}(\ell) = \pi_1(\ell) < \pi_1(k) = u^{-1}(k) \, \textrm{ if and only if } \, \pi_k(\ell) < \pi_k(k)$.
But $\pi_k(k)=p_k$, so 
\begin{equation*}
	 | \{ \ell>k : u^{-1}(\ell) < u^{-1}(k) \}| = | \{ \ell>k : \pi_k(\ell) < \pi_k(k)) \}| = p_k - k,
\end{equation*}
as desired.
\end{proof}

\subsubsection{\textsf{Tilings and wiring diagrams}}\label{sec:proof_characterization_theorem_dim2} 

The possible positions of the triangles in a lozenge 
tiling of $n\D_2$ naturally correspond to triples in the triangular array of
non-negative natural numbers $(x_1,x_2,x_3)$ whose sum is equal to $n-1$. The unit triangles at the corners $A, B,$ and $C$ have coordinates $(n-1,0,0)$, $(0, n-1,0)$, and $(0,0,n-1)$, respectively. We denote by $G_n$ the directed graph whose vertices are the triples
in this triangular array, and where each node which is not in the bottom row
is connected to the two nodes directly below it.  
There is a natural bijection between the lozenge tilings of $n\D_2$
and the vertex-disjoint routings to the bottom $n$ vertices of the graph $G_n$: 
simply place one rhombus over each edge in the routing, 
one vertical rhombus over each isolated vertex, and one 
triangle over the top vertex of each path in the routing  \cite{Luby}.
See Figure \ref{fig1_routing} for an example.

\begin{figure}[h]
	\centering
	\includegraphics[width=0.9\textwidth]{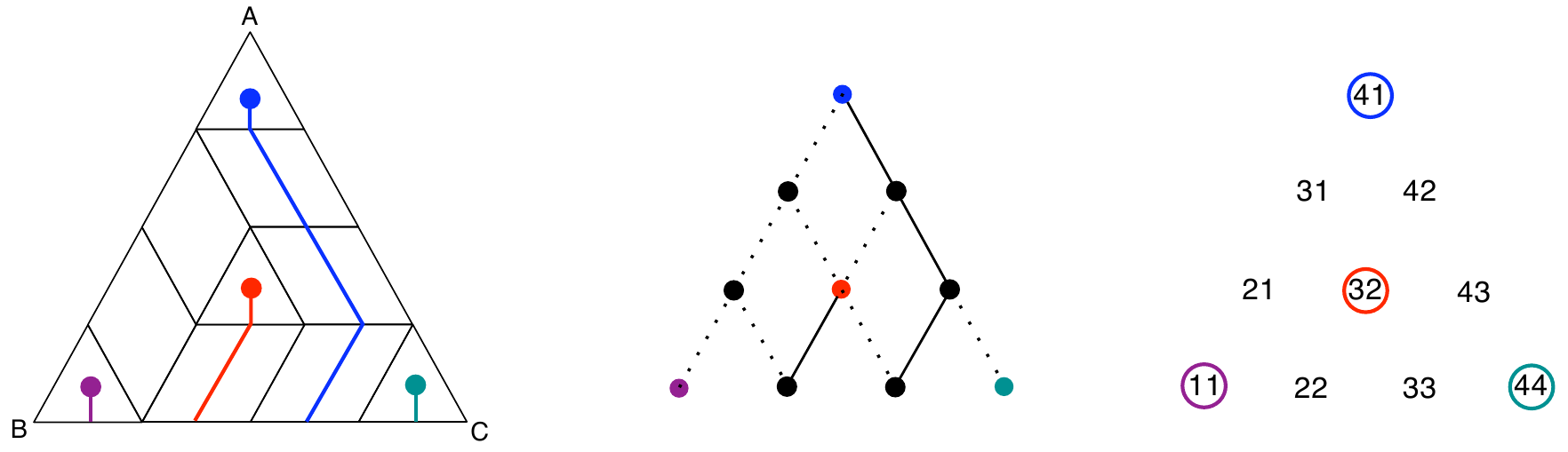}
	\caption{Left: A tiling of $4\D_2$. 
	Middle: The corresponding routing of $G_4$.  
	Right: The coordinates $(p,q)$ of the vertices of $G_4$.}
	\label{fig1_routing}
\end{figure}

 We perform a change of coordinates and label the nodes of $G_n$ with pairs of numbers $(p,q)$,
where $p=x_1+x_3+1$ and $q=x_3+1$. The $p$ and $q$ coordinates range from $1$ to $n$, and increase in the northeast and southeast directions, respectively.
Figure \ref{fig1_routing} shows the coordinates $(p,q)$ of the graph $G_4$. 
Given a lozenge tiling $T$, 
and the corresponding routing of $G_n$, 
number the positions of the unit triangles from $1$ up to $n$, such that
the vertex of the $i$th triangle is routed to the vertex with $(p,q)$-coordinate
equal to $(i,i)$. 
Let $(p_i, q_i)$ be the position of the $i$th triangle in $T$. Since the triangles of $T$ are spread out, we have $1 \leq q_i \leq i \leq p_i \leq n$ for all $i$. We now show that knowing the positions $(p_i, q_i)$ is equivalent to knowing the system of permutations $(u, v, w)$.

\begin{lem}\label{uv}
The ordered list of positions $(p_i, q_i)$ of the triangles in a lozenge tiling determines the permutations $u$ and $v$ as follows:
\[
\begin{array}{ccl}
	u	&=& (n,\dots ,p_n)\circ\dots \circ(2,\dots ,p_2)\circ(1,\dots ,p_1)\\
	v	&=& (1,\dots ,q_1)\circ\dots \circ(n-1,\dots ,q_{n-1})\circ(n,\dots ,q_n).\\
	w	&=& n\ldots321
\end{array}
\]
\end{lem}

\begin{proof}
The equation $w=n \ldots 321$ holds by assumption. We prove the formula for $u$; the proof for $v$ is analogous. The color $i$ splits naturally into three broken rays $R_A(i), R_B(i)$, and $R_C(i)$ centered at the $i$th triangle and pointing away from vertices $A, B,$ and $C$ respectively. Consider the pseudolines $L(i) = R_A(i) \cup R_C(i)$ for $1 \leq i \leq n$. We can regard this pseudoline arrangement as a wiring diagram for the permutation $u$. 

\begin{figure}[h]
	\centering
	\includegraphics[width=0.45\textwidth]{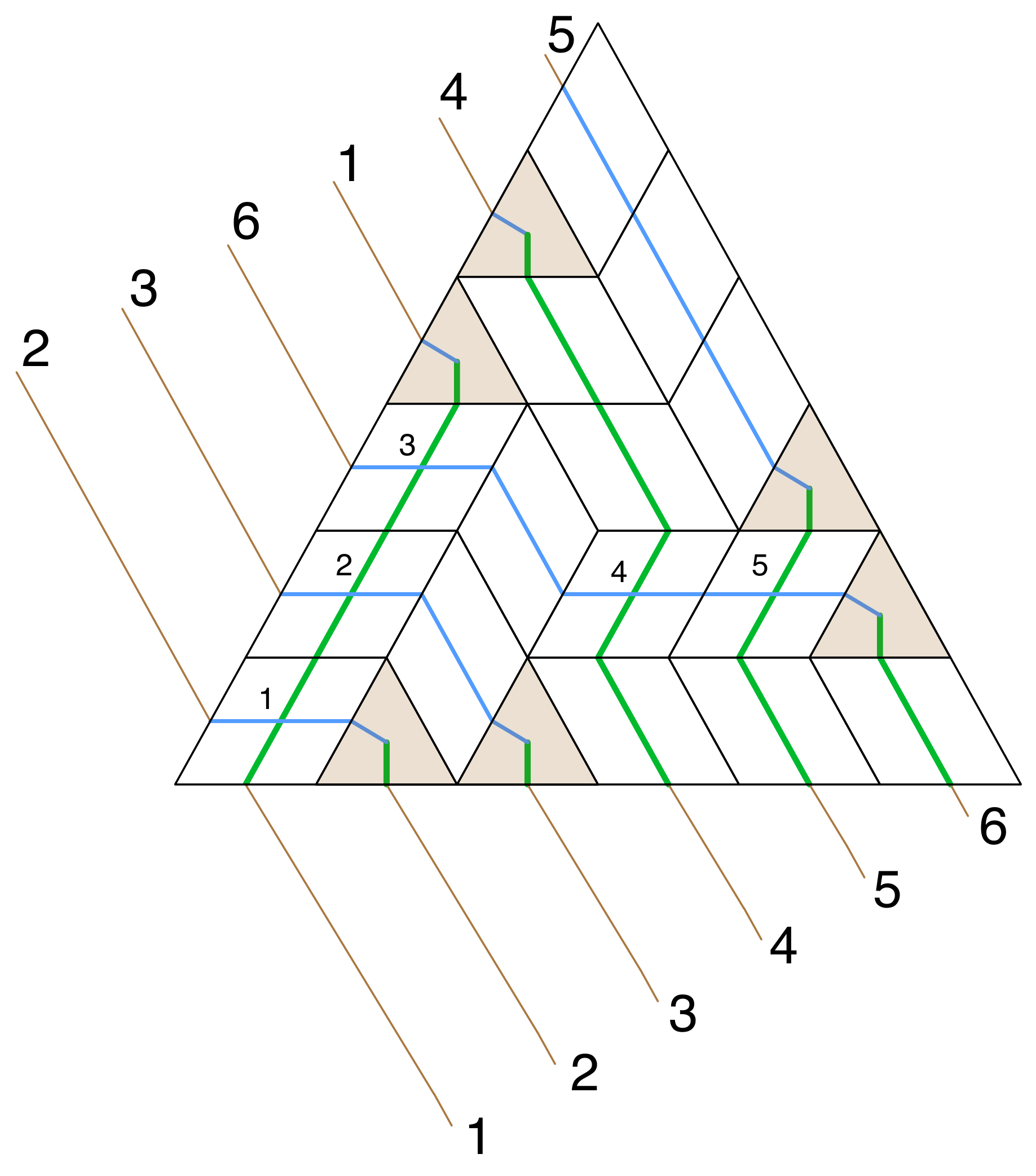}
	\caption{A tiling with $p=(4,2,3,5,6,6)$. We have $u=236145 = ()(s_5)(s_4)()()(s_1s_2s_3) = ({\bf 6})(5{\bf 6} )(4{\bf 5})({\bf 3})({\bf 2})(123{\bf 4})$.
In the other direction we have $q=(1,2,3,1,4,5)$ and $v=412563 =  ()()()(s_3s_2s_1)(s_4)(s_5) = (1{\bf 1})(2{\bf 2})(3{\bf 3})(432{\bf 1})(5{\bf 4})(6{\bf 5})$.}
	\label{fig:routing}
\end{figure}

To express $u$ as a product of transpositions, it suffices to linearly order the crossings from bottom to top, in an order compatible with the partial order given by the wiring diagram, and multiply them left to right. One way of doing it is to proceed up the ray $R_A(n)$, then up the ray $R_A(n-1)$, and so on up to the ray  $R_A(1)$, recording every crossing that we see along the way. This procedure lists every crossing exactly once, and the crossings along ray $R_A(i)$ correspond to the transpositions $s_i, s_{i+1}, \ldots, s_{p_i-1}$ (where $s_j=(j, j+1)$) which multiply to the cycle $(i, \dots, p_i)$.
\end{proof}

Notice that in Lemma \ref{uv}, we need the \textbf{ordered} list of positions of the triangles to determine the system of permutations. Figure \ref{fig:unordered} illustrates this.

\begin{figure}[h]
	\centering
	\includegraphics[width=0.4\textwidth]{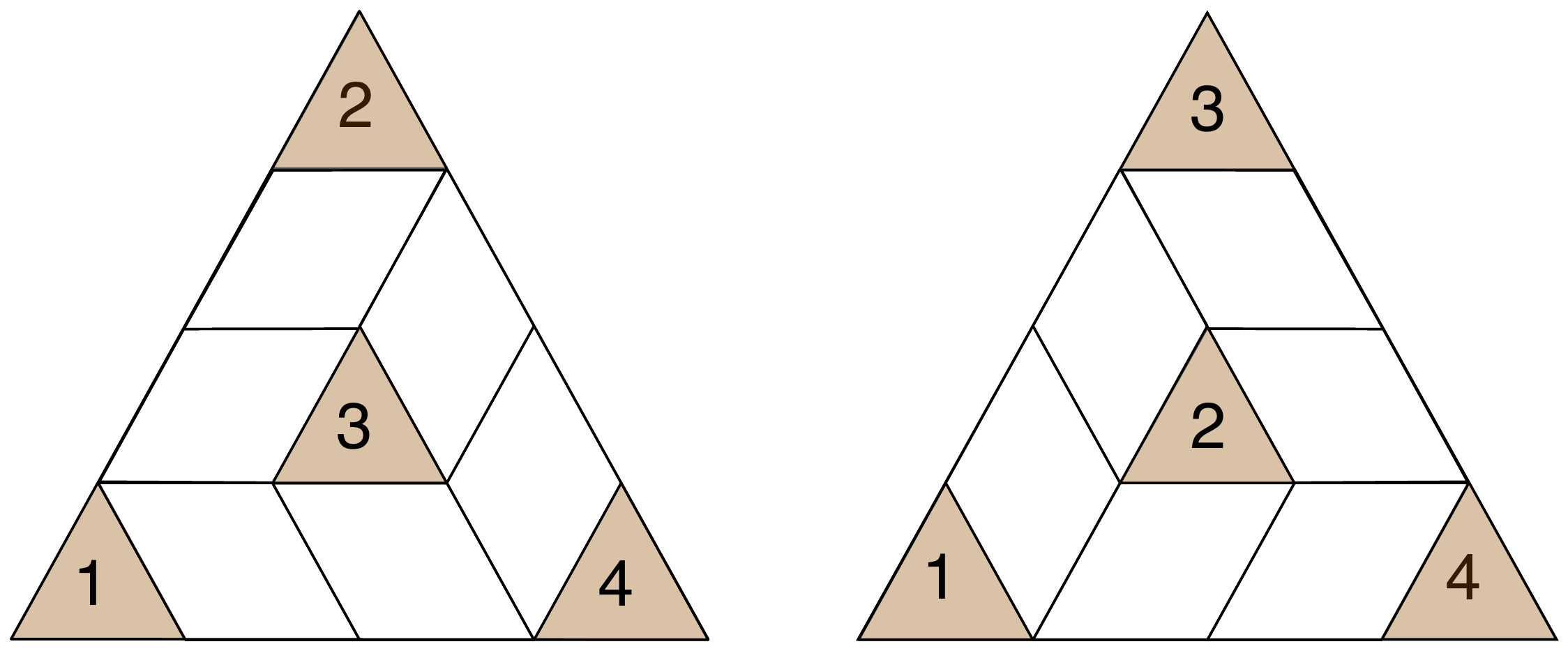}
	\caption{Two tilings with the same unordered set of triangles and different systems of permutations.}
	\label{fig:unordered}
\end{figure}

We have now done all the work to prove the main result of this section.

\begin{proof}[Proof of Theorem \ref{theo_permstoholes}]
If we are given the numbered triangle positions and a permutation of the system, we can assume without loss of generality that the given permutation is $w=n\dots 1$. Lemma \ref{uv} then tells us how to obtain the two remaining permutations $u$ and $v$. Moreover, given that $1 \leq q_i \leq i \leq p_i \leq n$, Lemmas \ref{perm.fact} and \ref{perm.fact2} imply that this procedure is reversible, and Lemma \ref{lem:positions_from_u_v} gives us an explicit way of computing the triangle positions in terms of $u$ and $v$. 
\end{proof}

\subsubsection{\textsf{From acyclic systems to triangle positions: another description}}

Let $T$ be a lozenge tiling of $n\Delta_2$ and $\sigma$ be its corresponding acyclic system of permutations. In addition to Lemmas \ref{lem:positions_from_u_v} \ref{uv}, we now present a different way of computing the triangle positions of $T$ in terms of $\sigma$.

\begin{figure}[h]
	\centering
	\includegraphics[width=1\textwidth]{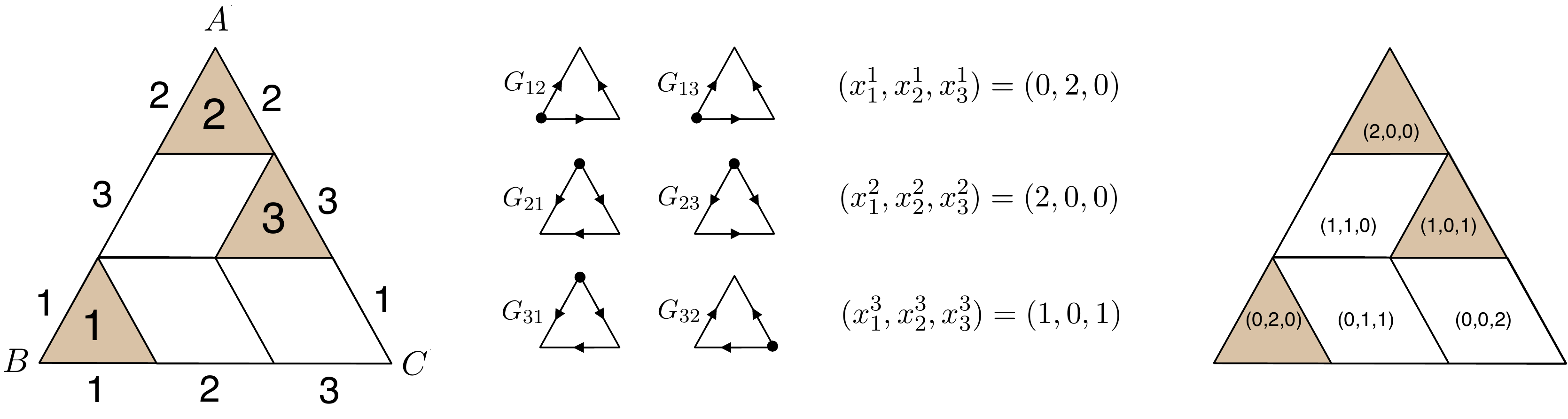}
	\caption{Obtaining the triangle positions from the system of permutations.}
	\label{fig:positions_dim2}
\end{figure}

As before, we identify the positions of the unit triangles in $T$ with triples in the triangular array of non-negative natural numbers $(x_1,x_2,x_3)$ whose sum is equal to $n-1$. For $1\leq i\neq j \leq n$ define the \emph{directed graph} $G_{ij}$ on the triangle $ABC$; we orient edge $e$ according to the order in which $i$ and $j$ appear on $e$ in the system of permutations $\sigma$.

Since the system of permutations is acyclic, each graph $G_{ij}$ is acyclic and has a unique source. The position $(x_1^i,x_2^i,x_3^i)$ of the $i$-th triangle is given by:
\begin{eqnarray*}
x_1^i=|\{j\neq i : \text{$A$ is the unique source of } G_{ij}\}|, \\
x_2^i=|\{j\neq i : \text{$B$ is the unique source of } G_{ij}\}|,\\
x_3^i=|\{j\neq i : \text{$C$ is the unique source of } G_{ij}\}|.
\end{eqnarray*}

We illustrate this in an example in Figure \ref{fig:positions_dim2}. This result is  proved in greater generality in Section \ref{sec:simplex_positions}.

\section{\textsf{The Acyclic System Conjecture}}\label{sec:acyclic_conjecture}

The main goal of this section is to introduce the concept of the system of permutations of a higher dimensional subdivision of a simplex. We prove that the system of permutations of a subdivision is acyclic. 
Motivated by Theorem~\ref{theo_characterization_dim2}, we conjectured that the converse statement holds as well (The Acyclic System Conjecture~\ref{conj_acyclic}). However, after this paper was submitted for publication, the conjecture was disproved by Francisco Santos in~\cite{Sa2012}.

Let $S$ be a subdivision on $n\D_{d-1}$.
As mentioned in Remark \ref{rem:notation}, in order to prevent confusion we will denote
indices in the set $[n]$ by the letters $i,j,k,\ell$, and indices in the set $[d]$ by the letters~$a,b$.
We also denote the vertices of the simplex $\D_{d-1}$ by
$w_1,\dots  , w_d$, the vertices of $n\D_{d-1}$ by
$nw_1,\dots  , nw_d$, the vertices of $\Delta_{n-1} $ by $v_1,\dots ,v_n$, 
and the vertices of d$\Delta_{n-1} $ by $dv_1,\dots ,dv_n$. 

The restriction $S |_{nw_aw_b}$ of the subdivision $S$ to the
edge $nw_aw_b$ is the subdivision of the segment $nw_aw_b$ given
by the Minkowski sums $S_1+\dots  + S_n \in S$ for which
$S_i\subset \{w_a,w_b, w_aw_b \}$ for all $i=1,\dots ,n$.

\begin{definition}[The permutation of an edge]
Since $S |_{nw_aw_b}$ is a subdivision, for each $i \in [n]$ there is a unique cell having $i-1$ summands equal to $w_b$, $n-i$ summands equal to $w_a$, and one summand (which we denote $S_{\sigma_{ab}(i)}$) equal to $w_aw_b$. It is easy to see that $\sigma_{ab}$ is a permutation of $[n]$, which we call \emph{the permutation of the edge $nw_aw_b$}. 
(Note that $\sigma_{ab}$ is the reverse of $\sigma_{ba}$ for any $1 \leq a \neq b \leq d$.) 
\end{definition}

\begin{remark} \label{rem:edge}
It is worth describing more explicitly the subdivision along each edge. As we traverse the edge $nw_aw_b$ from the vertex $nw_a$ to $nw_b$, the first edge of $S$ that we encounter has the form $w_a+\cdots+w_a + w_aw_b + w_a + \cdots + w_a$. Each subsequent edge is obtained from the previous one by converting the summand $w_aw_b$ into $w_b$ and converting one of the summands $w_a$ into $w_aw_b$. The permutation $\sigma_{ab}$ tells us the order in which the summands $w_a$ are converted to $w_aw_b$ (and then to $w_b$).
\end{remark}

\begin{definition}[The system of permutations]
The \emph{system of permutations of a fine mixed subdivision} 
$S$ of $n\D_{d-1}$ is the collection 
$\s(S)=(\sigma_{ab})_{1\le a \neq b \le d}$
of permutations 
$\sigma_{ab}$
of the edges $nw_aw_b$.
\end{definition}

\begin{example}
Consider the subdivision $S$ of $3\D_2$
given on the left hand side of Figure \ref{fine_mixed_voronoi},
with the small difference that we now call the vertices 
$3w_1, 3w_2$, and $3w_3$ instead of $A, B$, and $C$.
Writing only the one dimensional cells of the subdivision
restricted to the edges of the triangles we have:
\[ 
S|_{3w_1w_2} = S|_{3w_2w_1} = 
\{  
w_1w_2+w_1+w_1,\ w_2+w_1+w_1w_2,\ w_2+w_1w_2+w_2 
\} 
\]
\[ 
S|_{3w_2w_3} = S|_{3w_3w_2} = 
\{  
w_2+w_2w_3+w_2,\ w_2w_3+w_3+w_2,\ w_3+w_3+w_2w_3
\} 
\]
\[ 
S|_{3w_3w_1} = S|_{3w_1w_3} = 
\{  
w_3+w_1w_3+w_3,\ w_3+w_1+w_1w_3,\ w_1w_3+w_1+w_1.
\} 
\]
\\
The system of permutations $\s(S)=(\sigma_{ab})_{1\le a \neq b \le 3}$
is then given by 
\[
\begin{array}{ccccc}
 \sigma_{12}=132 &&&&   \sigma_{21}=231   \\
 \sigma_{23}=213 &&&&   \sigma_{32}=312   \\
 \sigma_{31}=231 &&&&   \sigma_{13}=132.   
\end{array}
\]
Notice that this system coincides with the restriction of the coloring of $S$ 
to the edges of the triangle. 
\end{example}

\begin{definition}
A \emph{system of permutations} on the edges of $n\D_{d-1}$
is a collection $\s=(\sigma_{ab})_{1\le a \neq b \le d}$
of permutations $\sigma_{ab}$ of $[n]$
such that
$\sigma_{ab}$ is the reverse of $\sigma_{ba}$ for all $a, b$. 
For each pair $1\le i \neq j  \le n$ we define the 
\emph{directed graph} $G_{ij}(\s)$ of $\s$
as the complete graph on $[d]$, where edge $ab$ is directed
$a\rightarrow b$ if and only if the permutation $\sigma_{ab}$
is of the form $\dots  i \dots  j \dots $. 
We say that a system of permutations $\s$ is \emph{acyclic}
if and only if all the graphs $G_{ij}(\s)$ are acyclic.
\end{definition}

\noindent In other words, a system of permutations 
on the edges of a simplex is acyclic if and only if there is no closed walk along the edges such that the permutation on every directed edge of the walk has the form $\dots  i \dots  j \dots$ for some $i$ and $j$.

\begin{theo}\label{theo_acyclic}
Let $S$ be a fine mixed subdivision of $n\Delta_{d-1}$, and $\s(S)$ be 
the corresponding system of permutations. Then $\s(S)$ is acyclic. 
\end{theo}

\begin{proof}
The cases $d=1, 2$ are trivial. The case $d=3$ was shown in Theorem \ref{theo_characterization_dim2}. 
For $d > 3$, notice that an orientation of the complete graph $K_d$ is acyclic if and only if every triangle $w_aw_bw_c$ is acyclic. But the orientation of triangle $w_aw_bw_c$ is given by the subdivision $S|_{nw_aw_bw_c}$, and so it is acyclic by Theorem \ref{theo_characterization_dim2}.
\end{proof}
We conjectured that the converse also holds:

\begin{perms_conj}\label{conj_acyclic}
Any acyclic system of permutations on the edges of the simplex
$n\Delta_{d-1}$ is achievable as the system of 
permutations of a fine mixed subdivision.
\end{perms_conj}

Theorem \ref{theo_characterization_dim2} says  that
the Acyclic System Conjecture \ref{conj_acyclic}  is true for $d=3$. 
However, Francisco Santos recently disproved this conjecture in the general case~\cite{Sa2012}. He constructed an acyclic system of permutations on the edges of $5\Delta_3$ that is not achievable as the system of permutations of a fine mixed subdivision.

\section{\textsf{Duality, deletion and contraction}}\label{section dual}
Before we continue extending the results of Section \ref{section system 2D} from two dimensions to higher dimensions we need some simple but useful machinery.
This section introduces the notion of duality, deletion and contraction 
for acyclic systems of permutations on the edges of a simplex.
We show that our definitions are compatible with the previously known
notions of duality, deletion and contraction for subdivisions 
\cite{[Ardila-Develin],[Cayley]}. Most of the results in this section follow
easily from the definitions, and we omit their proofs.

\subsection{\textsf{Duality for subdivisions}}\label{duality} 
%

There is a natural notion of duality between subdivisions of $n \Delta_{d-1}$ and subdivisions of $d \Delta_{n-1}$. This duality is induced by a one-to-one correspondence between subdivisions of $n\Delta_{d-1}$ and triangulations of the polytope $\Delta_{n-1}\times \Delta_{d-1}$ obtained via 
a special case of the Cayley trick~\cite{HRS2000}; for a more thorough discussion see \cite{HRS2000,[Cayley]}.
Figure \ref{cayley_T2} shows an example of a triangulation of the triangular prism $\Delta_1\times \Delta_2=12\times ABC$, and the corresponding subdivision of $2\Delta_2$ , whose three tiles are $ABC + B$, $AC + AB$, and $C + ABC$. 

\begin{figure}[h]
	\centering
	\includegraphics[width=0.6\textwidth]{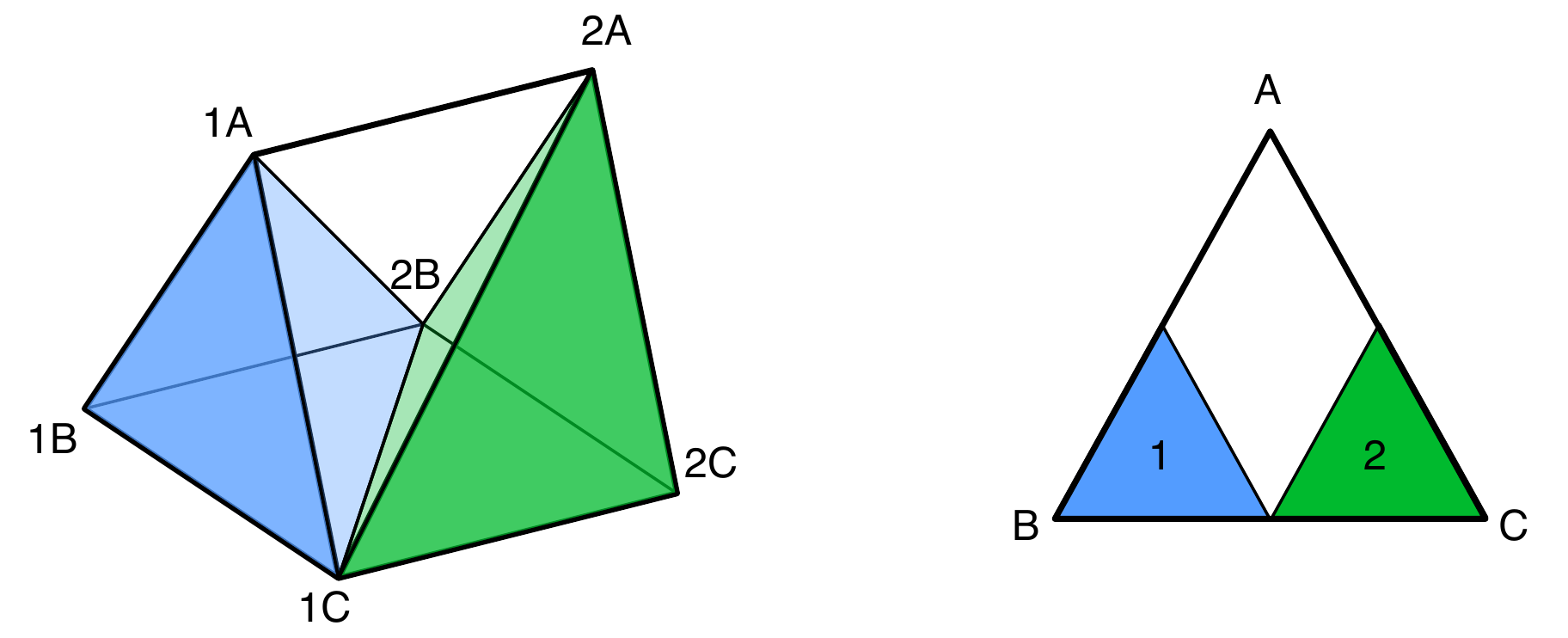}
	\caption{The Cayley trick}
	\label{cayley_T2}
\end{figure}


\noindent
Given a subdivision $S$ of $n\D_{d-1}$, 
we denote by $S^*$ the \emph{dual subdivision} of $d\D_{n-1}$ that corresponds to the triangulation of $\Delta_{d-1}\times \Delta_{n-1}$.
The dual subdivision is given by:
\begin{lem}\label{lem:dual_cell}
The dual of a mixed cell $S_1+\dots  +S_n$ in $S$ is the mixed cell 
$Z_1+\dots  + Z_d$ in $S^*$, 
where
$Z_a=\{v_i: w_a \in S_i \}$.
\end{lem}

Figure \ref{dual_subdivision} shows an example of a subdivision 
of $3\Delta _{4-1}$, its dual subdivision of $4\Delta_{3-1}$ and 
the Minkowski sum decompositions of the full dimensional cells.

\begin{figure}[h]
	\centering
	\includegraphics[width=0.7\textwidth]{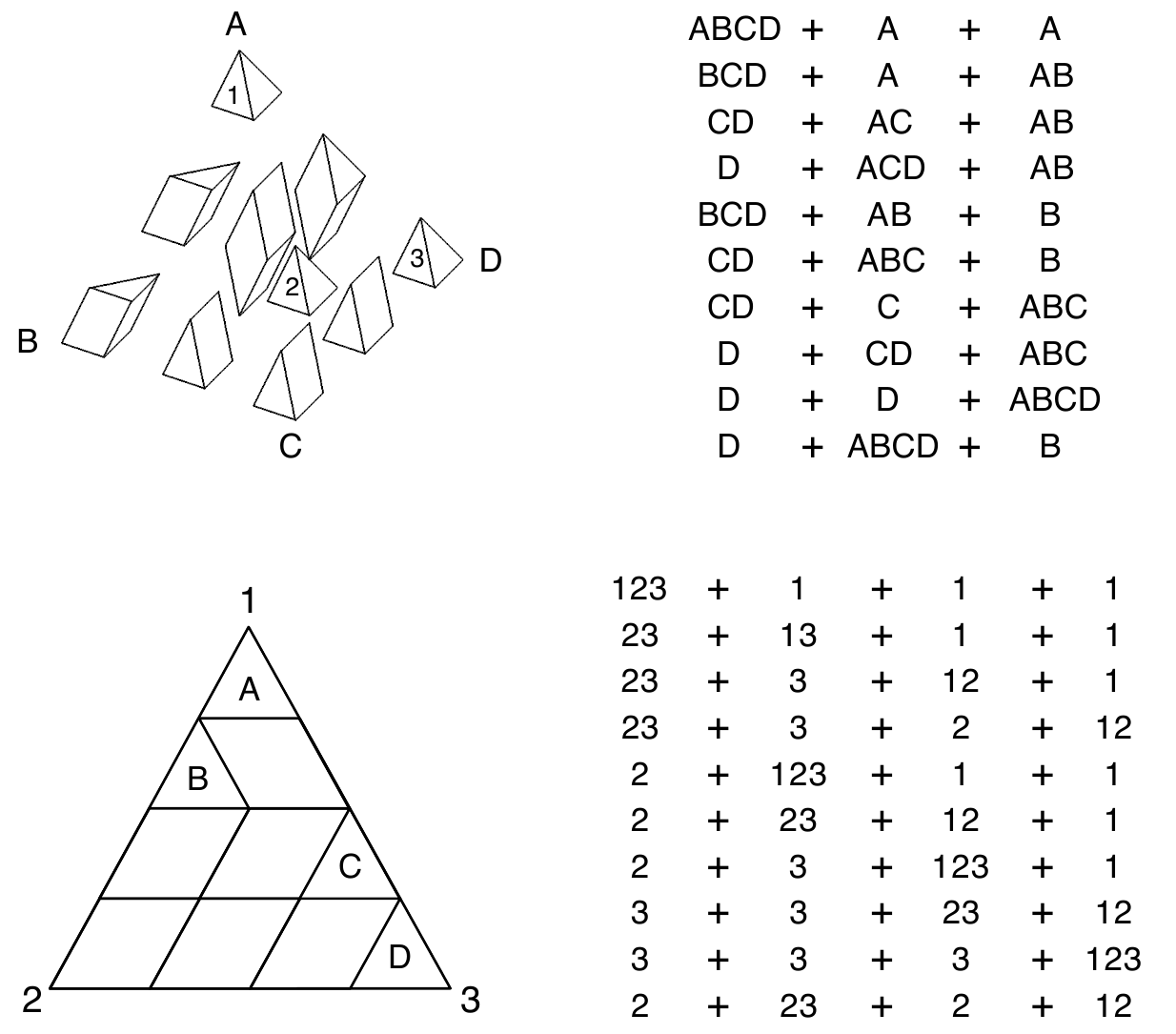}
	\caption{A subdivision $S$ of $3\Delta _{4-1}$, 
	its dual subdivision $S^*$ of $4\Delta_{3-1}$ 
	and the Minkowski sum decompositions of the full dimensional cells.  The systems of permutations $\s = \s(S)$ and $\s^* = \s(S^*)$ are given by 
$
	\s_{AB}=132,\ \s_{AC}=123,\ \s_{AD}=123, 
	\s_{BC}=123,\ \s_{BD}=123,\ \s_{CD}=123,
$
	and $\s^*_{12}=ABCD,\ \s^*_{23}=BACD,\ \s^*_{31}=DCBA$.
	\label{dual_subdivision}}
\end{figure}

\subsection{\textsf{Dualilty for acyclic systems of permutations}}

Let $\s= (\s_{ab})_{1\le a \neq b \le d}$ be an acyclic system 
of permutations on the edges of the simplex $n\D_{d-1}$.
Recall that, for each pair $1\le i \neq j \le n$, the graph $G_{ij}(\s)$
 on $[d]$ vertices
has a directed edge $a\rightarrow b$ if and only if
the permutation $\s_{ab}$ is of the form
$\dots  i \dots  j \dots $. 
Since this graph is acyclic and complete, 
it can be naturally regarded as a permutation $\s^*_{ij}$ of~$[d]$. 
More precisely, for $a\in [d]$, $\s^*_{ij} (a)$ is  the vertex of $G_{ij}(\s)$
whose out-degree is equal to $d-a$.

\begin{definition}
The \emph{dual system} 
of an acyclic system $\s$ of $n\D_{d-1}$
is the system of permutations 
$\s^*=(\s^*_{ij})_{1\le i \neq j \le n}$ 
on the edges of $d\D_{n-1}$.
\end{definition}

\begin{example}
If $\sigma_{12}=132, \sigma_{23}=213, \sigma_{31}=231$, then $\sigma^*_{12} = 132, \sigma^*_{23} = 231, \sigma^*_{31} = 321$.
\end{example}

\begin{lem}
The permutation $\s_{ab}$ is of the form $\dots  i \dots  j \dots $ if and only if 
the permutation $\s^*_{ij}$ is of the form $\dots  a \dots  b \dots $.
\end{lem}
\begin{proof}
Omitted.
\end{proof}

\begin{prop}
The system of permutations $\s^*$ is acyclic and $(\s^*)^*=\s$.
\end{prop}
\begin{proof}
Omitted.
%
\end{proof}

\subsection{\textsf{Deletion and contraction for subdivisions}}

Recall that we denote the vertices of $\Delta_{d-1}$ by $w_1,\dots ,w_d$, and the vertices of $\Delta_{n-1}$ by $v_1,\dots ,v_n$.

\begin{definition}
Let  $S$ be a subdivision of $n\Delta_{d-1}$. For $i\in [n]$, 
the \emph{deletion} $S{\backslash i}$ is the subdivision of $(n-1)\D_{d-1}$
whose mixed cells correspond to the Minkowski sums obtained 
from the Minkowski sums of $S$ by deleting the $i$th summand. 
\end{definition}


\begin{definition}
Let $S$ be a subdivision of the simplex $n\Delta_{d-1}$. 
For $a \in [d]$, the \emph{contraction} $S/a$ is the 
subdivision of $n\D_{d-2}$ which consists of the cells
$S_1+\dots  +S_n$ of $S$ such that 
$S_i\subset  \{ w_1,\dots ,\hat{w_a},\dots ,w_d \}$.
\end{definition}

It is useful to also think of the deletion $S{\backslash i}$ and the contraction $S{/a}$ in terms of the Cayley trick. The subdivision $S$ corresponds to a triangulation $T$ of $\Delta_{n-1}\times \Delta_{d-1}$; the deletion $S{\backslash i}$ and contraction $S{/a}$ correspond to the restriction of $T$ to the appropriate facets of $\Delta_{n-1}\times \Delta_{d-1}$.

\subsection{\textsf{Deletion and contraction for systems of permutations}}

\begin{definition}
Let $\s$ be a system of permutations on the edges of $n\Delta_{d-1}$. For $i\in [n]$, the \emph{deletion} $\s{\backslash i}$ is a system of permutations on the edges of $(n-1)\Delta_{d-1}$,  obtained from $\s$ by deleting the number $i$ from each permutation.
\end{definition}

\noindent For example, for the system of permutations 
$\s=\{1423, 3124, 4321\}$ 
on the edges of $4\D_3$ in Figure \ref{Coloration of a triangle}, we get 
$\s{\backslash 2}=\{143, 314, 431\}$. 
Note that $\s{\backslash 2}$ is the system of permutations of the deletion $S{\backslash2}$ of the subdivision $S$ in the same figure.

\begin{definition}
Let $\s$ be a system of permutations on the edges of $n\Delta_{d-1}$. For $a \in [d]$, the \emph{contraction} $\s{/a}$ is the restriction of $\s$ to the edges of the facet of $n\Delta_{d-1}$ 
that do not contain the vertex $nw_a$. 
\end{definition}

The following proposition follows directly from the definitions.

\begin{prop}
If $\s=\s(S)$ is the system of permutations of a subdivision $S$, 
then 
\begin{enumerate}
\item $\s{\backslash i}=\s (S \backslash i)$
\item $\s{/a}=\s(S/a)$
\end{enumerate}
\end{prop}


\subsection{\textsf{Properties}}
In this subsection we show that the operations of deletion and contraction
are dual to each other, and that the dual system of the system of
permutations of a subdivision $S$ is equal to the system of permutations
of the dual subdivision $S^*$. This result will be a key lemma in Section \ref{Section_Applications}.

\begin{prop}\label{prop_dual1}
Let $S$ be a subdivision of $n\D_{d-1}$ and $\s$ be an acyclic system of permutations of $n\Delta_{d-1}$. Let $i\in [n]$ and $a \in [d]$. Then
\[
(S{\backslash i})^*=S^*{/i}, \qquad 
(S{/a})^*=S^*{\backslash a},\qquad 
(\s{\backslash i})^*=\s^*{/i}, \qquad 
(\s{/a})^*=\s^*{\backslash a}.\qquad 
\]

\end{prop}

%
%

\noindent The geometric content of this propositions is the following:
The deletion of a color $i$ in a subdivision $S$ corresponds
to the contraction of the vertex $i$ in the dual subdivision $S^*$, and
viceversa.
This proposition follows directly from the Cayley trick and from the definitions.

\begin{prop}\label{dual_conf=conf_of_dual}
Let $S$ be a subdivision of $n\Delta_{d-1}$ and $\sigma(S)$ be the associated system of permutations. Then ${\s(S)}^*=\s(S^*)$.
\end{prop}
\begin{proof}
It suffices to check the relative positions of any two numbers $i,j$ on the edges of $n\Delta_{d-1}$; \emph{i.e.}, to prove the result for $n=2$. Now, all mixed subdivisions of $2\Delta_{d-1}$ are isomorphic, up to relabeling, to the one whose $d$ full-dimensional cells are:
\[
w_1w_2w_3 \ldots  w_ d + w_1, \,
\,w_2w_3 \ldots  w_ d + w_1w_2, \,
\,w_3 \ldots  w_ d + w_1w_2w_3, \,
\ldots, 
\,w_d + w_1w_2w_3\ldots w_d.
\]
The proposition is easily verified in this case.
\end{proof}

The operations of restriction, contraction, and duality will be very be useful to us in what follows.

\section{\textsf{From systems of permutations to simplex positions}} \label{sec:simplex_positions}

A subdivision of $n\Delta_{d-1}$ is not uniquely determined by its system of permutations, but in this section we will see that the positions of its simplices \textbf{are} completely determined. We already proved this result for lozenge tilings in Theorem \ref{theo_permstoholes}. We now prove it in general.

\begin{theo}[Acyclic systems of permutations and simplex positions]\label{theorem of positions}
The numbered positions of the simplices in a fine mixed subdivision $S$ of~$n\Delta_{d-1}$ are completely determined by its system of permutations $\s = \s(S)$. More precisely, the Minkowski 
decomposition of the $i$th simplex is
\[
w_{a_1} + \cdots + w_{a_{i-1}} + w_1w_2\ldots w_n + w_{a_{i+1}}+ \cdots + w_{a_n}
\]
where $a_j$ is the unique source of the acyclic graph $G_{ij}(\s)$ for all $j\neq i$.
\end{theo}

\begin{proof}
The Minkowski sum decomposition of the $i$th simplex $T$ in the subdivision $S$ 
has $i$th component equal  to $w_1\ldots w_d$. By Lemma~\ref{lem:dual_cell},
this implies that the letter $v_i$ appears in all the components of the Minkowski decomposition of the dual cell $T^*$. It follows that $T^*$ is the unique full dimensional cell of the subdivision $S^*$ 
that contains the vertex $dv_i$ of $d\Delta_{n-1}$.
This dual cell $T^*$  is completely determined by the colors of the edges adjacent 
to the vertex $dv_i$ in $S^*$, which are precisely the sources of 
$G_{ij}(\s)$ with $j\neq i$. 
More explicitly, for each index $a \in [d]$, the Minkowski decomposition of the dual cell $T^*$ has  $a$th component  equal to $\{ v_j,\ j\in [n] : j=i \text{ or source}(G_{ij}) = a \}$.
Dualizing and applying Lemma~\ref{lem:dual_cell} again, we get the desired Minkowski decomposition of the simplex $T$. 
\end{proof}

\begin{figure}[h]
	\centering
	\includegraphics[width=0.7\textwidth]{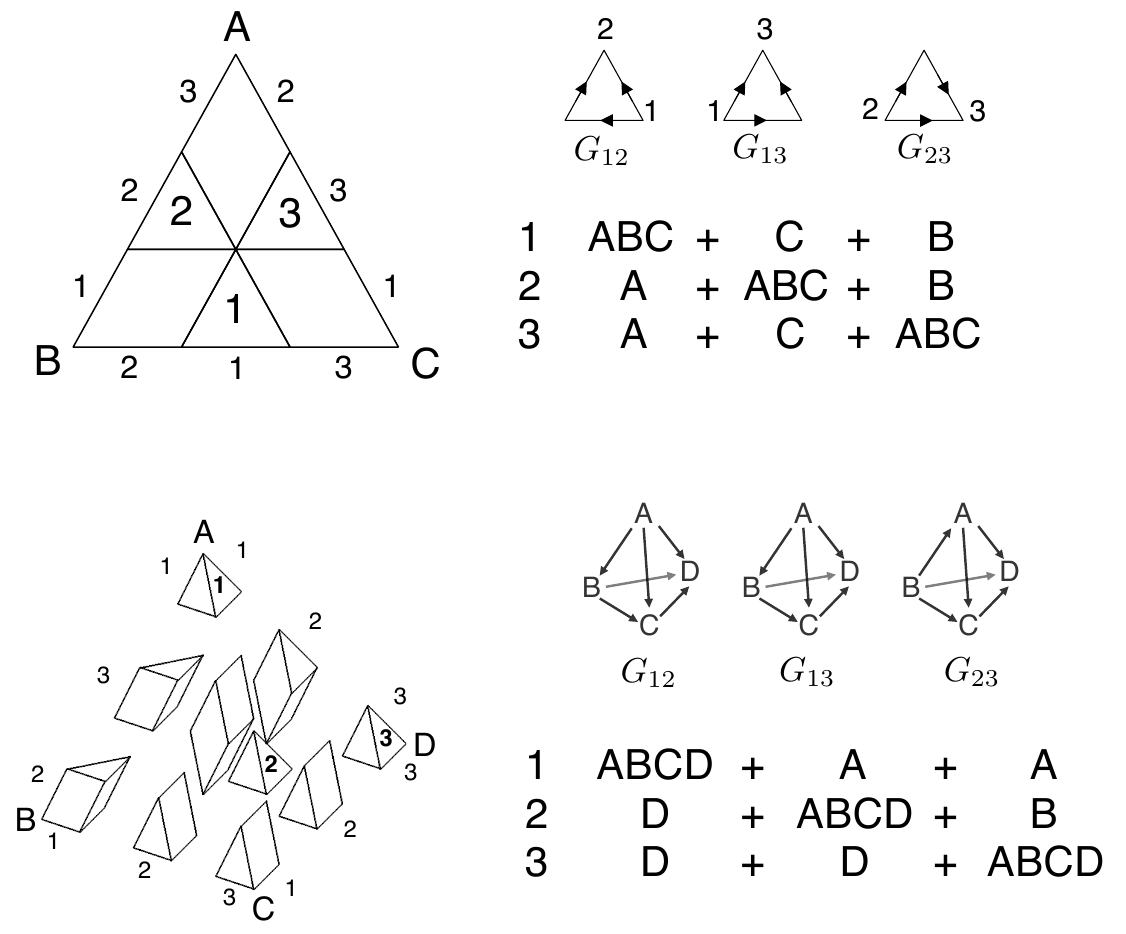}
	\caption{
	How to obtain the positions (and Minkowski sum decompositions) of the simplices from the
	system of permutations. 	} 
	\label{orientation_T3}
\end{figure}

Figure \ref{orientation_T3} shows two examples of how to compute the positions
of the simplices in a subdivision from its system of permutations.

\begin{remark}
In dimension 2, our proofs of Theorem \ref{theo_permstoholes} and Theorem \ref{theorem of positions} give two descriptions of the triangle positions as a function of the acyclic system of permutations. The first description, and
Lemma \ref{lem:positions_from_u_v} in particular,  gives a simple computation of these positions. The second description gives us additional information about the triangles, namely, their Minkowski sum decompositions. 
For higher $d$, this second computation works without modification. It would also be interesting to generalize the first one; \emph{i.e.}, to find a direct description of the positions of the simplices in the spirit of our proof of Theorem \ref{theo_permstoholes}.
\end{remark}

\section{\textsf{The Spread Out Simplices Conjecture}}\label{Section_Applications}

One of the motivations of this paper is the Spread Out Simplices Conjecture of Ardila and Billey. They showed that every subdivision of $n\Delta_{d-1}$ contains precisely $n$ unit simplices, and studied where those simplices could be located. They conjectured that the possible positions of the simplices are given by the bases of the matroid determined by the lines in a generic complete flag arrangement. For details, see \cite{[Ardila-Billey]}.

The following is an equivalent statement. Recall that a collection of $n$ simplices in $n \Delta_{d-1}$ is said to be \emph{spread out} if no subsimplex of size $k$ contains more than $k$ of them.

\begin{holes_conj}\cite[Conjecture 7.1]{[Ardila-Billey]}.
 \label{holes_conjecture}
 A collection of $n$ simplices in $n\Delta_{d-1}$ can be extended to a fine mixed subdivision if and only if it is spread out.
\end{holes_conj}

\subsection{\textsf{Regular subdivisions: Yoo's example}}\label{sec:yoo}

Question 8.3 in \cite{[Ardila-Billey]} asked whether 
Conjecture \ref{holes_conjecture} is true in the more restrictive context of \textbf{regular} subdivisions. One may ask the same question for 
Conjecture \ref{conj_acyclic}. Hwanchul Yoo \cite{Yoo} showed that these statements are  false in that context, even for $d=3$. Figure \ref{fig:non_regular_system} shows an acyclic system of permutations on $6 \Delta_2$ and a collection of 6 triangles which can only be realized by two subdivisions, neither of which is regular.

\begin{figure}[h]
	\centering
	\includegraphics[width=0.3\textwidth]{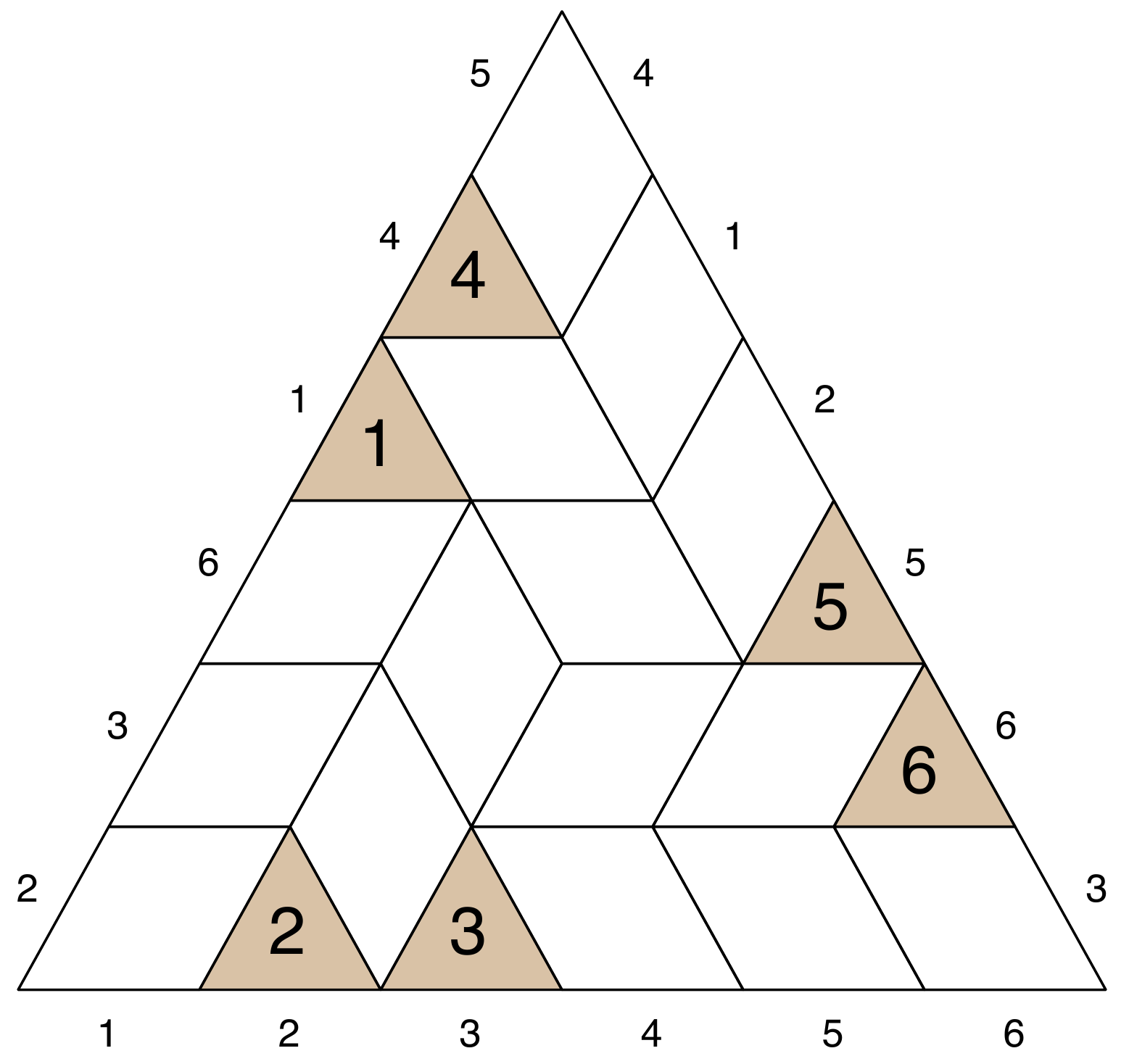}
	\caption{A non-regular subdivision of $6\D_2$. 
	Its system of permutations and triangle positions
	cannot be achieved by a regular subdivision. 
	(Example by Hwanchul Yoo).}
	\label{fig:non_regular_system}
\end{figure}

\subsection{\textsf{The simplex positions are spread out}}

We identify the possible positions of the simplices in a 
subdivision 
with the lattice points of the simplex
$\{
(x_1,\dots ,x_d) \in \mathbb R^d : 
x_1+\dots  + x_d = n-1 \text{ and } x_1,\dots ,x_d \ge 0
\}.
$
Theorem \ref{theorem of positions}
leads us to the following definition.

\begin{definition}
\emph{The set of simplex positions} $P(\s)$ of an acyclic
system of permutations $\s$ of $n\D_{d-1}$ is as follows: For $1 \leq i \leq n$ the $i$th simplex has position $(x_1^i, \ldots, x_d^i)$, where 
\[
x_a^i=|\{j\neq i : \text{$a$ is the unique source of } G_{ij}(\sigma)\}|.
\]
\end{definition}

\noindent  Notice that this definition makes sense for arbitrary acyclic systems, and not
only for those coming from subdivisions. When $\s$ comes from a subdivision, the set of positions $P(\s)$ is the one given by Theorem \ref{theorem of positions}. 

\begin{remark} \label{simplicesviadual}
As seen in the proof of Theorem \ref{theorem of positions}, computing the positions of the simplices of $\s$ is very easy if we know $\s^*$. Each simplex of $\s$ corresponds to a vertex $v$ of $d \Delta_{n-1}$, and its Minkowski summands are $\Delta_{d-1}$ and the $d-1$ labels on the edges coming out of $v$ in $\s^*$.  For instance, if $\s$ is the subdivision on the top of Figure \ref{dual_subdivision}, then its three simplices are readily given by the permutations around the dual triangle in $\s^*$: they are $ABCD+A+A, D+ABCD+B, D+D+ABCD$. 
\end{remark}

Recall that Ardila and Billey proved the 
forward direction of 
Conjecture \ref{holes_conjecture}:

\begin{theo}[{\cite[Proposition 8.2] {[Ardila-Billey]}}]\label{coro}
The positions of the simplices of a fine mixed subdivision
of $n\D_{d-1}$ are spread out.
\end{theo}

In principle, our next result generalizes this:

\begin{theo}\label{alternative_proof}
The positions of the simplices $P(\s)$ of an acyclic system of permutations $\s$
of $n\D_{d-1}$ are spread out.
\end{theo}



In order to prove Theorem \ref{alternative_proof}, let us define 
the \emph{table of positions} $T(\s)$ of an acyclic system of permutations $\s$.
This is an $n\times n$ matrix whose 
rows are given by the Minkowski summands of the simplices of
the system of permutations $\s$.  For example, the table of positions of the system of permutations $\{123,231,312\}$
on the top example in Figure \ref{orientation_T3} is 
$$\begin{bmatrix}
ABC&C&B\\
A&ABC&B\\
A&C&ABC
\end{bmatrix}$$

For each $a \neq b \in [d]$, define the directed graph $H_{ab}(\s)$ as the graph on the vertex set $[n]$ containing a directed edge $i\rightarrow j$ if there is a $w_a$ in 
 row $i$ and a $w_b$ in  row $j$ which are in the same column. 
 In the previous example, $H_{AB}(\sigma)$ is a complete graph in three vertices with directed edges $2\rightarrow 1$, $3\rightarrow 1$ and $3\rightarrow 2$. Notice that in this case $H_{AB}(\sigma)$ is the complete acyclic graph that corresponds to the permutation $\sigma_{AB}=321$. In general we have:

\begin{lem}\label{positions to system}
The graph $H_{ab}(\s)$ is a subgraph of the complete, acyclic graph on $[n]$ that corresponds to the permutation $\s_{ab}$.
\end{lem}

\begin{proof}
We need to prove that if $H_{ab}$ has a directed edge $i\rightarrow j$ 
then the permutation $\s_{ab}$ is of the form $\dots  i\dots  j\dots  $. 
Suppose $i\rightarrow j$ in $H_{ab}$. Then there is a column $\ell $ of the table $T(\sigma)$ 
such that $w_a\in T_{i\ell }$ and $w_b\in T_{j\ell }$.\\
\textit{Case 1:} If $\ell =j$, then $T_{ij}=w_a$. This means that the graph $G_{ij}$ has a source at $a\in [d]$, which implies that the permutation $\sigma_{ab}$ is of the form $\dots  i\dots  j\dots  $.\\
\textit{Case 2:} If $\ell =i$, an analogous argument works.\\
\textit{Case 3:} If $\ell \neq i,j$, then $T_{i\ell }=w_a$ and $T_{j\ell }=w_b$, which means that the graphs $G_{i\ell }$ and $G_{j\ell }$ have sources at $a\in[d]$ and $b\in [d]$ respectively. This implies that the permutation $\sigma_{ab}$ is of the form $\dots  i\dots  \ell \dots  j\dots$.
\end{proof}

\begin{remark}
The graph $H_{ab}$ is in general only a proper subgraph of the graph that corresponds to $\sigma_{ab}$. In particular, in contrast with the 2-dimensional case, the ordered list of positions of the simplices (or even their Minkowski sum decompositions) is not sufficient to determine the system of permutations. For instance, there are exactly two subdivisions of $2\Delta_3$ for which the Minkowski decomposition of the two unit simplices are $ABCD+A$ and $B+ABCD$. These two subdivisions have different systems of permutations. The graph $H_{CD}$ in these cases consists of two vertices $C$ and $D$ without any edge.
\end{remark}

\begin{proof}[Proof of Theorem \ref{alternative_proof}]
Let $\s$ be an acyclic system of permutations on the edges of $n\D_{d-1}$,
$P(\s)$ be the set of positions of the simplices and $T(\s)$
be the table of positions.
Suppose that there is a sub-simplex $\D$ of size $k$ containing 
more than $k$ simplices of $\s$. 
This sub-simplex is given by
\[
\D =\{ x=(x_1\ldots x_d)\in \mathbb 
R^d : x_i \geq m_i \text{ and } x_1+\ldots x_d =1 \},
\] 
for some non-negative integers $m_1,\dots  ,m_d$ such that 
$m_1+\dots  +m_d=n-k$.  
Without loss of generality, we assume that the
(more than $k$) simplices of $\s$ that are contained in $\D$
correspond to the first rows of the table $T(\s)$. 
Each one of these rows contains (off of the diagonal) 
at least $m_a$ letters $w_a$ for all $a \in [d]$.
Call such a letter \emph{dark} if it is in the shaded rectangle 
in Figure \ref{Table}, and \emph{light} if it is in the white square
on the upper left. This shaded rectangle has width less than $n-k$.

\begin{figure}[h]
	\centering
	\includegraphics[width=0.5\textwidth]{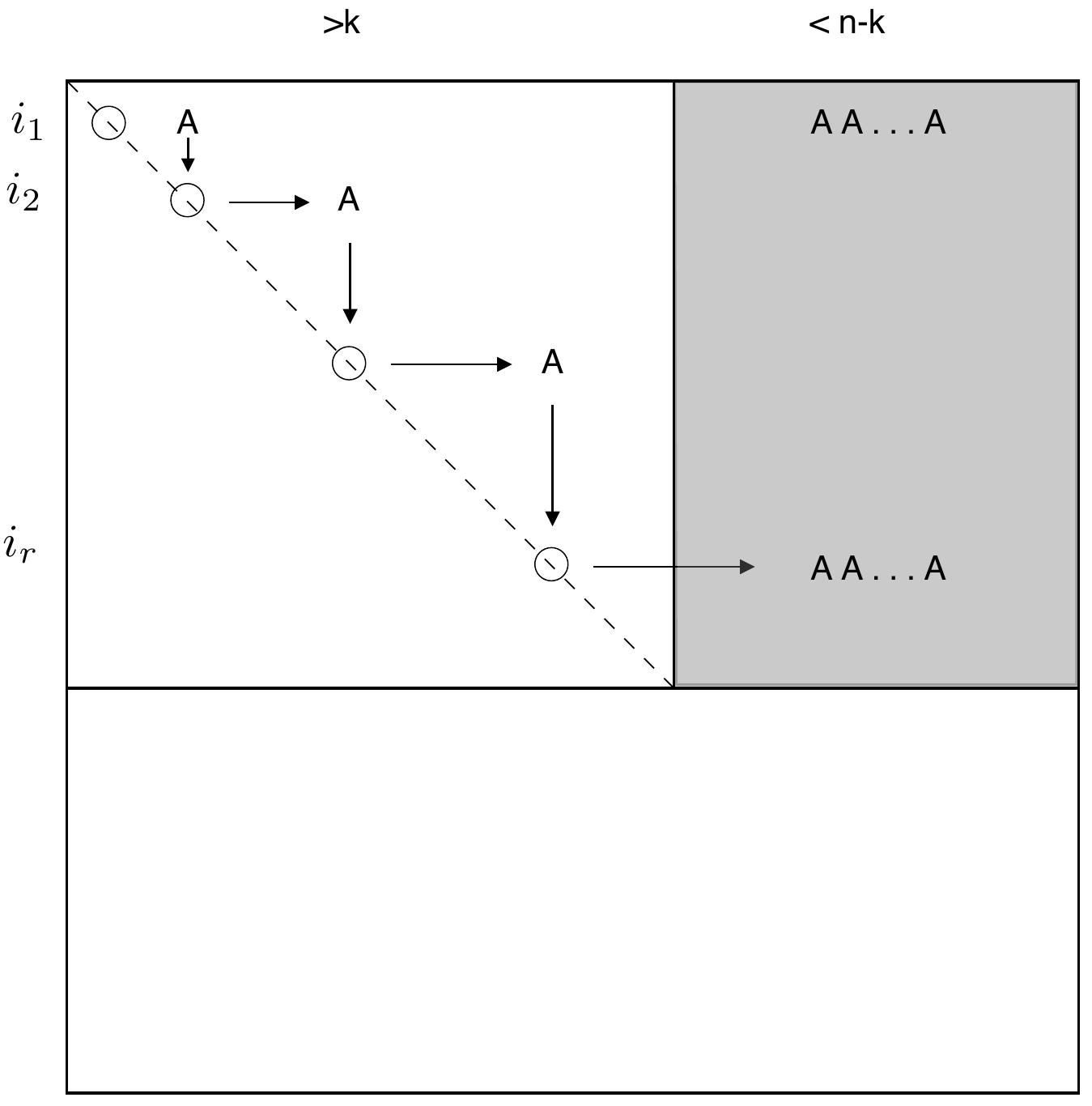}
	\caption{Table of positions $T(\s)$}
	\label{Table}
\end{figure}

\noindent We will prove that the first row of $T(\s)$
has at least $m_a$ dark letters $w_a$ for all $a\in [d]$.
If $m_a=0$ the result is obvious.
Now suppose $m_a>0$.
If there is no light letter $w_a$ on the first row of $T(\s)$
then the claim clearly follows. Otherwise, we will
construct, simultaneously for all $b \in [d] \backslash \{a\}$,
a path
$1=i_1\rightarrow \dots  \rightarrow i_r$
in the graphs $H_{a b}$
ending on a row $i_r$ that has no light letter $w_a$. To do so, we start by drawing the arrows $1=i_1\rightarrow i_2$ in $H_{a b}$
for all $b$, 
where $i_2$ is the column of the first letter $w_a$ under consideration.
If there is no light letter $w_a$ in row $i_2$, we are done; otherwise, we continue the process.
Since the graphs $H_{a b}$ have no cycles, this
process must end at some row $i_r$. 

Notice that row $i_r$ must contain at least $m_a$ dark letters, since it contains no light ones except for the one on the diagonal.
Now consider the letters on the first row which are directly above the dark letters
$w_a$ on the row $i_r$. They must all be equal to $w_a$, or else they would 
form a cycle 
$1=i_1\rightarrow \dots  \rightarrow i_r \rightarrow 1$
in $H_{ab}$ for some $b$.
Therefore, the first row of $T(\s)$
has at least $m_a$ dark letters $w_a$ as we claimed.

Finally, observe that the first row contains at least $m_a$ dark letters $w_a$ for all $a \in [d]$, so the shaded rectangle must have width at least $m_1 + \dots + m_d = n-k$, a contradiction.
\end{proof}

We conjecture that the converse of Theorem \ref{alternative_proof} is true as well:

\begin{weakholes_conj}\label{system_holes_conjecture}
Any $n$ spread out simplices in $n\D_{d-1}$ can be achieved as the simplices of an acyclic system of permutations.
\end{weakholes_conj}

This conjecture is weaker than the Spread Out Simplices Conjecture \ref{holes_conjecture},
and has the advantage that it is more tractable computationally for small 
values of $n$. Conjecture \ref{holes_conjecture} would follow from 
Conjecture \ref{conj_acyclic} and 
Conjecture \ref{system_holes_conjecture}.
In the next Section we prove these three conjectures in the special
case $n=3$.

As we remarked earlier, Santos~\cite{Sa2012} recently gave an example of an acyclic system of permutations on the edges of $5\Delta_3$ which does not extend to a mixed subdivision. This is a counterexample to the Acyclic System Conjecture \ref{conj_acyclic}. By Theorem \ref{alternative_proof}, Santos's example gives a 5 spread out simplices in $5\Delta_3$, which would be a good candidate for a counterexample of the Spread Out Simplices Conjecture \ref{holes_conjecture}. However, as Santos remarks, these 5 spread out simplices can be extended to a mixed subdivision of $5 \Delta_3$. The Spread Out Simplices Conjecture \ref{holes_conjecture} (and its Weak version, Conjecture \ref{system_holes_conjecture}) remain open.


\subsection{\textsf{The Spread Out Simplices Theorem for simplices of size three
}}

\begin{theo}[Acyclic System Conjecture \ref{conj_acyclic} for $n=3$]
\label{size3 cesar's conjecture}
Every acyclic system of permutations on $3\Delta_{d-1}$ 
is achievable as the system of permutations of a fine mixed subdivision.
\end{theo}
\begin{proof}
Let $\s$ be an acyclic system of $3\Delta_{d-1}$ and $\s^*$ be the dual
system of $d\D_2$. By
Theorem \ref{theo_characterization_dim2} 
there exist a subdivision $S^*$ of $d\D_2$ whose 
system of permutations is equal to $\s^*$. 
The system of permutations of the 
dual subdivision $S=(S^*)^*$ of $3\Delta_{d-1}$ is $\s$.
\end{proof}

\begin{theorem}[Weak Spread Out Simplices Conjecture \ref{system_holes_conjecture} for $n=3$]
Any three  spread out simplices in $3\D_{d-1}$ 
are achievable as the simplices of an acyclic system of permutations.
\end{theorem}

\begin{proof}
The position of a simplex  of an acyclic system of permutations
of $3\D_{d-1}$ corres\-ponds to a Minkowski sum of the form 
$w_1\dots  w_d + w_{a_1}+w_{a _2}$. We identify the position of 
such a simplex with the pair of letters $w_{a_1}w_{a_2}$.
For simplicity, we denote the letters $w_1,\dots ,w_d$ by  by $A,B,\dots ,H$.
Figure \ref{size3} lists all the possible (combinatorial types of) triples of pairs of letters which correspond to positions of spread out simplices. (For instance $AB, AC, AD$ is missing because these would correspond to three simplices in a simplex of size $2$.)
\begin{figure}[h]
	\centering
	\includegraphics[width=0.9\textwidth]{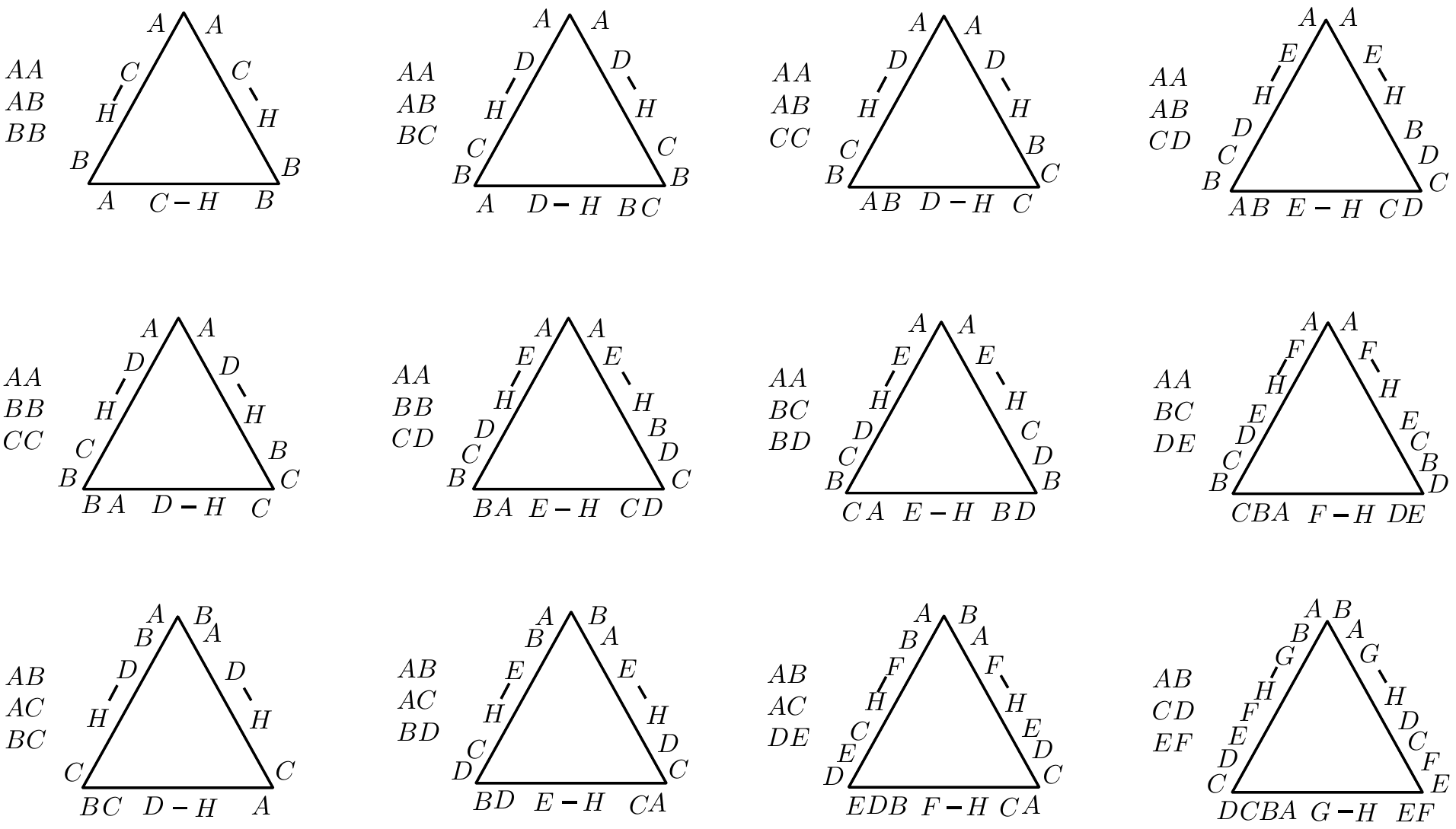}
	\caption{Combinatorial types of the positions of any three spread out simplices in 
	$3\Delta_{d-1}$, and (non-unique) dual acyclic systems which generate such positions.}
	\label{size3}
\end{figure}

By duality and Remark \ref{simplicesviadual}, we will be done if, for each such triple, we can build a dual acyclic system $\s^*$ of $d\D_2$ such that these are the pairs of labels adjacent to each of the three vertices of $d\D_2$. This is also done in Figure \ref{size3}. The duals to those acyclic systems 
give rise to the desired simplex positions.
\end{proof}

\begin{theo}[Spread Out Simplices Conjecture \ref{holes_conjecture} for $n=3$]
\label{size 3 federico's conjecture}
Three simplices in $3\Delta_{d-1}$ can be extended to a fine mixed subdivision if and only if they are spread out.
\end{theo}
\begin{proof}
This result is a consequence of the previous two theorems.
\end{proof}

\bigskip

\noindent\thanks{\textsf{\textbf{Acknowledgments:}} 
We would like to thank Hwanchul Yoo for his example in Section \ref{sec:yoo}, and Christopher O'Neill for running some computational experiments in support of the Weak Spread Out Simplices Conjecture \ref{system_holes_conjecture}.   
We are also grateful to Carolina Benedetti and Erik Backelin for helpful
discussions about this paper. 
}

\end{document}